\documentclass[]{article}

\usepackage[utf8]{inputenc}
\usepackage[margin=1in]{geometry}

\usepackage{amsthm, amsfonts, amsmath, makecell, tikz, 
 graphicx, colortbl, mathtools}
\usepackage{multirow, booktabs, bigstrut}

\usepackage{graphicx}

\usepackage{float}
\usepackage{color}
\usepackage{tikz,pgfplots}
\usepackage{tikz-cd}
\usepackage{verbatim}
\usepackage{enumerate}
\usepackage{physics}
\usepackage{braket}

\usepackage{bm}

\usepackage{enumerate, amssymb}

\usetikzlibrary{graphs,graphs.standard,arrows.meta,calc,intersections}

\usepackage[]{algorithm, algorithmic}

\usepackage{tabularx}

\theoremstyle{plain}
\newtheorem{theorem}{Theorem}[section]
\newtheorem{lemma}[theorem]{Lemma}
\newtheorem{corollary}[theorem]{Corollary}
\newtheorem{proposition}[theorem]{Proposition}

\newtheorem{remark}[theorem]{Remark}

\theoremstyle{definition}
\newtheorem{definition}[theorem]{Definition}
\newtheorem{example}[theorem]{Example}

\DeclareMathOperator{\Int}{Int}

\DeclareMathOperator{\Pic}{Pic}

\DeclareMathOperator{\IntR}{Int{}^\text{R}}

\DeclareMathOperator{\minval}{minval}
\DeclareMathOperator{\loc}{loc}

\DeclareMathOperator{\Spec}{Spec}

\renewcommand{\epsilon}{\varepsilon}

\newcommand{\R}{{\mathbb R}}
\newcommand{\N}{{\mathbb N}}

\newcommand{\Q}{{\mathbb Q}}

\newcommand{\Z}{{\mathbb Z}}

\newcommand{\p}{\mathfrak{p}}

\newcommand{\m}{\mathfrak{m}}
\newcommand{\M}{\mathfrak{M}}

\definecolor{gray}{rgb}{.5,.5,.5}

\definecolor{black}{rgb}{0,0,0}

\definecolor{blue}{rgb}{0,0,1}

\definecolor{red}{rgb}{1,0,0}

\definecolor{green}{rgb}{0,1,0}

\definecolor{gold}{rgb}{.5,.5,.2}

\definecolor{yellow}{rgb}{1,1,.4}

\definecolor{purple}{rgb}{.5,0,.5}

\definecolor{darkgreen}{rgb}{0,.5,0}

\definecolor{orange}{rgb}{1,.55,0}

\definecolor{white}{rgb}{1,1,1}

\usepackage[colorinlistoftodos]{todonotes}

\let\originalleft\left
\let\originalright\right
\renewcommand{\left}{\mathopen{}\mathclose\bgroup\originalleft}
\renewcommand{\right}{\aftergroup\egroup\originalright}

\let\originaltodo\todo
\renewcommand{\todo}[1]{\originaltodo[inline]{#1}}

\usepackage{hyperref}
\usepackage{xcolor}
\hypersetup{
	colorlinks,
	linkcolor={red!50!black},
	citecolor={blue!50!black},
	urlcolor={blue!80!black}
}

\usepackage{subfigure}

\title{Ring Structure of Integer-Valued Rational Functions}

\author{Baian Liu}

\begin{document}

\maketitle
\begin{abstract}
	Integer-valued rational functions are a natural generalization of integer-valued polynomials. Given a domain $D$, the collection of all integer-valued rational functions over $D$ forms a ring extension $\IntR(D)$ of $D$. For a valuation domain $V$, we characterize when $\IntR(V)$ is a Prüfer domain and when $\IntR(V)$ is a Bézout domain. We also extend the classification of when $\IntR(D)$ is a Prüfer domain.
\end{abstract}
\section{Introduction}

\indent\indent Integer-valued polynomials appear in many areas of mathematics, including Hilbert polynomials of polynomial rings and interpolation formulas. A natural generalization of integer-valued polynomials is integer-valued rational functions. One of the first analyses of integer-valued rational functions was done by Kochen \cite{Kochen}. Kochen uses rings of integer-valued rational functions to help determine when Diophantine equations have an integral solution over a $p$-adically closed field. 

We are interested in studying sets of integer-valued rational functions over a ring as a ring itself. We will be investigating the ring-theoretic properties of these rings of integer-valued rational functions on their own and in relation to the ring over which they are defined.

 We start with a domain $D$. Then we want to define the ring of integer-valued rational functions $\IntR(D)$ over $D$, as well as some notions to help us investigate $\IntR(D)$. 

\begin{definition}
	Let $D$ be a domain and $K$ its field of fractions. We define the \textbf{ring of integer-valued rational functions over $D$} to be
	\[
	\IntR(D) = \{ \varphi \in K(x) \mid \varphi(D) \subseteq D\}.
	\]
\end{definition}

\begin{remark}
	The set $\IntR(D)$ is in fact a ring which is also closed under composition.
\end{remark}

Requiring $\varphi(D) \subseteq D$ means that $\varphi$ cannot have any poles in $D$. Since a rational function only has finitely many poles, we many choose to ignore these finitely many elements. This turns out to not change the set of rational functions we are considering, since if $D$ is not a field, a rational function $\varphi \in K(x)$ such that $\varphi(d) \in D$ for all but finitely many $d \in D$ is in $\IntR(D)$ \cite[p. 260]{Cahen}.

We can also obtain a ring by taking the collection of all rational functions that are integer-valued on some subset of the field of fractions. These rings can help give a more nuanced description of what properties of $D$ lead to certain properties of $\IntR(D)$.

\begin{definition}
	Take $D$ to be a domain and $K$ its field of fractions. Let $E$ be some subset of $K$. We can more generally define the \textbf{ring of integer-valued rational functions on $E$ over $D$} to be
	\[
	\IntR(E,D) = \{ \varphi \in K(x) \mid \varphi(E) \subseteq D\}.
	\]
\end{definition}

In particular, for the ring $\IntR(K,D)$, we may also choose to ignore finitely many values in determining if a rational function is integer-valued on $K$. 

\begin{proposition}
	Let $D$ be a domain that is not a field and $K$ its field of fractions. If $\varphi \in K(x)$ is such that $\varphi(r) \in D$ for almost all $r \in K$, then $\varphi \in \IntR(K, D)$. 
\end{proposition}

\begin{proof}
	Let $b \in D$ such that $b \neq 0$. Define $\psi_b(x) = \varphi\left(\frac{x}{b}\right)$. By assumption, $\psi_b$ is almost integer-valued on $D$, so $\psi_b$ is integer-valued on $D$. Now let $d \in K$. We can write $d = \frac{a}{b}$ for some $a, b \in D$ with $b \neq 0$. Then $\varphi(d) = \psi_b(a) \in D$. Thus, $\varphi \in \IntR(K,D)$. 
\end{proof}

The definition of rings of integer-valued rational functions appears to be very similar to that of rings of integer-valued polynomials. Despite this, we will see that the behavior of rings of integer-valued rational functions can be vastly different from that of rings of integer-valued polynomials. For example, rings of integer-valued rational functions are not as sensitive to infinite residue fields. In the case of integer-valued polynomials, if $D$ is a domain with infinite residue fields, then $\Int(D) = D[x]$ \cite[p. 10]{Cahen}. 

One notion that marks a difference between integer-valued rational functions and polynomials is the notion of unit-valued polynomials. Take $D$ to be a domain. Because the inverse of a nonconstant polynomial is not a polynomial, unit-valued polynomials are not units in $\Int(D)$. However, the inverse of a polynomial is a rational function, so unit-valued polynomials are units in $\IntR(D)$.

\begin{definition}
	Let $D$ be a domain and let $K$ be the field of fractions of $D$. Then a polynomial $f \in D[x]$ is \textbf{unit-valued} over $D$ if $f(D)\subseteq D^\times$. We will denote the set of all unit-valued polynomials by
	\[
	T = \{f\in D[x] \mid f(D) \subseteq D^\times \}.
	\]
\end{definition}

\begin{remark}
	Let $\Int(D)$ denote the ring of unit-valued polynomials of $D$. We can see that $T$ is a multiplicative subset of $\Int(D)$ and we have the containment
	\[
	T^{-1}\Int(D) \subseteq \IntR(D). 
	\]
\end{remark}

The containment $T^{-1}\Int(D) \subseteq \IntR(D)$ can be strict (see Example \ref{Ex:NotJustLocalizationByUVPolynomials}), so unit-valued polynomials do not explain all of the differences between $\Int(D)$ and $\IntR(D)$. Nevertheless, unit-valued polynomials are a useful tool for describing some of the structure of $\IntR(D)$. In addition, unit-valued polynomials over a domain $D$ are closely linked to the residue fields of $D$, so residue fields of $D$ can help describe the structure of $\IntR(D)$. 

\begin{proposition}\label{Prop:UVNoRoots}
	Let $D$ be a local domain with maximal ideal $\mathfrak{m}$. Then $f \in D[x]$ is unit-valued if and only if $f \mod \mathfrak{m}$ has no roots in $D/\mathfrak{m}$. 
\end{proposition}

\begin{proof}
	We have that $f \in D[x]$ is unit-valued if and only if $f(d) \notin \mathfrak{m}$ for any $d \in D$, which happens if and only if $f \mod \mathfrak{m}$ has no roots in $D/\mathfrak{m}$.

\end{proof}

\begin{corollary}
	Let $D$ be a domain. Then $f \in D[x]$ is unit-valued if and only if for every maximal ideal $\m$ of $D$, the polynomial $f \mod \m$ has no roots in $D/\m$.
\end{corollary}

\begin{proof}
	If $f$ is unit-valued over $D$, then $f$ is valued in $D\setminus \m$ for each maximal ideal $\m$, so $f \mod \m$ has no roots in $D/\m$. Conversely, if $f$ is unit-valued over $D_\m$ for every maximal ideal $\m$ of $D$, then $f(a), \frac{1}{f(a)} \in D$ for each $a \in D$, which means $f$ is unit-valued over $D$. 
\end{proof}

There exist domains $D$ such that $\IntR(D) = \Int(D)$. In other words, all of the integer-valued rational functions over $D$ are polynomials. This means analyzing $\IntR(D)$ is no different than analyzing $\Int(D)$, so we are only interested in domains $D$ such that $\IntR(D) \neq \Int(D)$. One observation we make is that if $D$ is a domain such that there exists a nonconstant unit-valued polynomial $f$, then $\frac{1}{f} \in \IntR(D) \setminus \Int(D)$. A ring with such a property is called a non-D-ring.

\begin{definition}
	Let $D$ be a domain. We call $D$ a \textbf{non-D-ring} if there is a nonconstant unit-valued polynomial $f \in D[x]$. 
\end{definition}

\begin{example}
	Let $D$ be a domain with a nonzero Jacobson radical, such as a semi-local domain or a local domain. Then $D$ is a non-D-ring. To see this, take $d$ to be a nonzero element in the Jacobson radical. Then $dx + 1$ is a nonconstant unit-valued polynomial. 
\end{example}

The Jacobson radical of $D$ need not to be nonzero for $D$ to be a non-D-ring, as seen in the following example.

\begin{example}\cite[Example 1.11]{LoperWithout}
	The domain $D = \Z\left[\frac{1}{p} \,\middle\vert\, p \equiv 1,2 \pmod 4, \text{$p$ is a prime}\right]$ has unit-valued polynomial $x^2+1$ and is therefore a non-D-ring. 
\end{example}

The property of being a non-D-ring turns out to be exactly the one we want to consider in order to study rings of integer-valued rational functions. 

\begin{proposition}\cite[Proposition 1]{GunjiMcQuillan}
	Let $D$ be a domain. Then $D$ is a non-D-ring if and only if $\IntR(D) \neq \Int(D)$. 
\end{proposition}

In general, the ring $\IntR(D)$ is not the localization of $\Int(D)$ by unit-valued polynomials. The following is an example of a domain $D$ such that $\IntR(D) \neq T^{-1}\Int(D)$. 

\begin{example}\label{Ex:NotJustLocalizationByUVPolynomials}
	Let $V$ be a valuation domain with an infinite residue field and a principal maximal ideal, generated by some $t \in V$. Then we claim that $\varphi(x) \coloneqq \frac{t}{x^2+t} \in \IntR(V) \setminus T^{-1}\Int(V)$. 
	
	Let $v$ be the valuation associated with $V$. Take $d \in V$. If $v(d) = 0$, then $v(\varphi(d)) = v(t) \geq 0$. If $v(d) > 0$, then $v(\varphi(d)) = 0$. Thus, $\varphi \in \IntR(V)$. 
	
	Now suppose that $\varphi = \frac{f}{g}$, where $f \in \Int(V)$ and $g \in T$. Since $V$ is local with infinite residue field, we have $\Int(V) = V[x]$. Then we obtain
	\[
		g = \frac{x^2 + t}{t} \cdot f.
	\]
	Let $d \in V$ such that $v(d) = 0$. Evaluate at $x=d$ to get $g(d) = \frac{d^2+t}{t} \cdot f(d)$. We see that $v(f(d)) = v(t)$, so $f(d) = 0 \mod (t)$. Since $V/(t)$ is infinite, we must have that $f(x) \mod (t)$ is the zero polynomial. However, evaluating $x = 0$, we get that $g(0) = f(0)$, so $f(0)$ is a unit and $f(x) \mod (t)$ cannot be the zero polynomial, a contradiction. Thus, $\varphi \notin T^{-1}\Int(V)$. 
\end{example}

We can study the structure of $\IntR(E, D)$ via its ideals. In particular, we can analyze its prime and maximal ideals. Since $\IntR(E,D)$ consists of functions valued in $D$, we can define some of the ideals of $\IntR(E,D)$ via ideals of $D$. We call these pointed ideals. 

\begin{definition}
	Let $D$ be a domain with field of fractions $K$. Take $E$ to be a subset of $K$. Also let $I$ be an ideal of $D$ and $a \in E$. Then define
		\[
	\mathfrak{I}_{I, a} = \{\varphi \in \IntR(E,D) \mid \varphi(a) \in I \}.
	\]
	Ideals of $\IntR(E,D)$ this form are called \textbf{pointed ideals}. If $I$ is a prime ideal $\p$ of $D$, we use the notation $\mathfrak{P}_{\p,a} = \mathfrak{I}_{\p,a}$ and call these \textbf{pointed prime ideal}. 
	
	If $\m$ is a maximal ideal of $D$, then we use the notation $\M_{\m, a}$ for $\mathfrak{P}_{\m, a}$. We call ideals of $\IntR(E,D)$ of this form \textbf{pointed maximal ideals}. 
\end{definition}

\begin{remark}
	The notation $\mathfrak{P}_{I, a}$ does not indicate the ring $D$ and subset $E$, so $D$ and $E$ are understood from context.
	
	Note that if $\p$ is a prime ideal of $D$, then $\mathfrak{P}_{\p, a}$ is a prime ideal of $\IntR(E, D)$, so it is justified to call $\mathfrak{P}_{\p, a}$ a pointed prime ideal. Moreover, the pointed maximal ideal $\M_{\m, a}$ is indeed a maximal ideal of $\IntR(E,D)$. 
\end{remark}

\begin{proposition}
	Let $D$ be a domain and $E$ a subset of the field of fractions. If $\p$ is a prime ideal of $D$, then for any $a \in E$, we have $\IntR(E,D)/\mathfrak{P}_{\p, \alpha} \cong D/\p$. 
\end{proposition}

\begin{proof}
	Consider the map $\IntR(E,D) \to D/\p$ given by $\varphi \mapsto \varphi(a) \mod \p$. This map is surjective since the constant functions are in $\IntR(E,D)$. Furthermore, the kernel of this map are rational functions in $\IntR(E,D)$ such that their evaluation at $a$ modulo $\p$ is 0, so the kernel is exactly $\mathfrak{P}_{\p, \alpha}$. Thus, $\IntR(E,D)/\mathfrak{P}_{\p, \alpha} \cong D/\p$. 
\end{proof}

\begin{remark}
	In particular, for any maximal ideal $\m$ of $D$, we have $\IntR(E,D)/\mathfrak{M}_{\m, a} \cong D/\m$, implying that $\mathfrak{M}_{\m, a}$ is a maximal ideal in $\IntR(E,D)$. 
\end{remark}

 In general, the pointed ideals of $\IntR(E,D)$ are not sufficient to describe all of the ideals of $\IntR(E,D)$. Also, in general, the pointed prime ideals do not describe all of the prime ideals either, and even the pointed maximal ideals do not describe all of the maximal ideals in general. However, we can use ultrafilters to describe more of the ideals. We first give the definition of an ultrafilter. 

\begin{definition}
	Let $S$ be a set. An \textbf{ultrafilter} $\mathcal{U}$ on $S$ is a set of subsets of $S$ such that
	\begin{itemize}
		\item $\emptyset \notin \mathcal{U}$,
		\item if $A \subseteq B \subseteq S$ and $A \in \mathcal{U}$, then $B \in \mathcal{U}$,
		\item if $A, B \in \mathcal{U}$, then $A \cap B \in \mathcal{U}$,
		\item if $A \subseteq S$, then $A \in \mathcal{U}$ or $S \setminus A \in \mathcal{U}$. 
	\end{itemize}
	
	Fix an element $s \in S$. The collection of all subsets of $S$ containing $s$ is an ultrafilter. We call ultrafilters of this form \textbf{principal} and \textbf{non-principal} otherwise. 
\end{definition}

\begin{remark}
	If $S$ is an infinite set, there exist non-principal ultrafilters on $S$ by Zorn's Lemma.
\end{remark}

Now we use ultrafilters to take ultrafilter limits of a set of prime ideals.

We can define a prime ideal from a collection of prime ideals using ultrafilters in a more general setting. The general construction does not limit us to rings of integer-valued rational functions.

\begin{definition}
	Let $R$ be a commutative ring and $\{\p_\lambda\}_{\lambda \in \Lambda}$ a subset of $\Spec(R)$. Consider an element $a \in R$. The \textbf{characteristic set of $a$ with respect to $\{\p_\lambda\}$} is defined as
	\[
		\Phi_a \coloneqq \{\p_\lambda \mid a \in \p_\lambda \}.
	\] 
	Take $\mathcal{U}$ to be an ultrafilter of $\{\p_\lambda\}$. We define the \textbf{ultrafilter limit of $\{\p_\lambda\}$ with respect to $\mathcal{U}$} as
	\[
		\lim\limits_{\mathcal{U}} \p_\lambda \coloneqq \{a \in R \mid \Phi_a \in \mathcal{U} \}. 
	\]
\end{definition}

\begin{remark}
	Since any ultrafilter on $\{\p_\lambda\}$ can be extended uniquely to an ultrafilter on $\Spec(R)$ containing $\{\p_\lambda\}$ and any ultrafilter on $\Spec(R)$ containing $\{\p_\lambda\}$ can be obtained this way, we may take $\mathcal{U}$ to be an ultrafilter of $\{\p_\lambda\}$ or an ultrafilter of $\Spec(R)$ containing $\{\p_\lambda\}$. 
\end{remark}

Using the definition of an ultrafilter, we may confirm that $\lim\limits_{\mathcal{U}} \p_\lambda$ is a prime ideal. In particular, if we take $\{\p_\lambda\}$ to be a set of pointed prime ideals of $\IntR(E,D)$, then the ultrafilter limit of $\{\p_\lambda\}$ with respect a non-principal ultrafilter can yield a prime ideal of $\IntR(E,D)$ that is not a pointed prime ideal.

For rings of integer-valued polynomials, we have $S^{-1}\Int(D) \subseteq \IntR(S^{-1}D)$ for any multiplicative subset $S$ of a domain $D$ \cite[Proposition I.2.2]{Cahen}. However, for rings of integer-valued rational functions, we don't necessarily have inclusion of $S^{-1}\IntR(D)$ in $\IntR(S^{-1}D)$ for $S$ a multiplicative subset of $D$. 

\begin{example}
	Let $k$ be a field and let $K = k(s,t)$ with a valuation $v:K \to \Z \oplus \Z \cup \{\infty\}$	given by
	\[
		v\left(\sum_{i,j} a_{ij} s^i t^j \right) = \min\limits_{ i,j} \{ (i,j)\}
	\]
	for each nonzero $\sum_{i,j} a_{ij} s^i t^j \in k[s,t]$, where each $a_{ij} \in k \setminus \{0\}$, and extended uniquely to $K$. The value group is ordered lexicographically. Let $D$ be the associated valuation domain. Its prime spectrum is $(0) \subsetneq (s, s/t, s/t^2, s/t^3, \dots) \subsetneq (t)$. If $S = D \setminus (s, s/t, s/t^2, s/t^3, \dots) $, then $\frac{1/t}{x-1/t} \in \IntR(D)$ but $\frac{1/t}{x-1/t} \notin \IntR(S^{-1}D)$. 
\end{example}

In Section 2, we discuss rings of integer-valued rational functions over valuation domains. For a valuation domain $V$, we completely determine when $\IntR(V)$ is a Prüfer domain. We also completely determine when $\IntR(V)$ is a Bézout domain. When $\IntR(V)$ is not a Prüfer domain, we determine prime ideals that are not essential. 

In Section 3, we consider integer-valued rational functions over a Prüfer domain $D$. We give some conditions when $\IntR(D)$ is not Prüfer and a family of Prüfer domains such that $\IntR(D)$ is Prüfer for each domain $D$ in this family.

\section{Integer-valued rational functions over valuation domains}

\indent\indent In this section, we let $V$ be a valuation domain, $K$ its field of fractions, $\m$ the maximal ideal of $V$, $v$ the associated valuation, and $\Gamma$ the value group. We investigate whether $\IntR(V)$ is a Prüfer domain or not. Prüfer domains are of particular interest since they possess many nice properties, such as satisfying a generalized version of the Chinese Remainder Theorem or being a generalized notion of a Dedekind domain. Prüfer domains can also be seen as a global version of a valuation domain. 

Not only are valuations powerful tools that assist in analyzing the ring of integer-valued rational function, but valuations also induce a topology that interacts well with integer-valued rational functions. The following proof is a modification of \cite[Proposition X.2.1]{Cahen}.

\begin{proposition}\label{Prop:RationalFunctionsContinuousInValuationTopology}
	Let $D$ be a domain with field of fractions $K$ and $E$ a subset of $K$. Let $v:K^\times \to \Gamma$ be a valuation such that the induced valuation ring $V$ contains $D$. Then each element of $\IntR(E,D)$ is a continuous function from $E$ to $D$ with respect to the topology induced by the valuation. 
\end{proposition}

\begin{proof}
	Let $\varphi \in \IntR(E,D)$. We can write $\varphi = \frac{f}{g}$ for $f, g \in K[x]$ relatively prime polynomials. Fix $a \in E$, and let $b \in E$ and $\varepsilon \in \Gamma$. We calculate
	\[
	\varphi(b) - \varphi(a) = \frac{f(b)}{g(b)} - \frac{f(a)}{g(a)} = \frac{f(b)-f(a)}{g(a)} - \frac{f(b)}{g(b)}\cdot \frac{g(b)-g(a)}{g(a)}. 
	\]
	Since $g(a) \neq 0$, we can say that $v(g(a)) = \gamma$ for some $\gamma \in \Gamma$. Let $\delta \in \Gamma$. Since $f$ and $g$ are continuous with respect to the topology induced by the valuation, there is some $\delta \in \Gamma$ such that $v(a-b) > \delta$ implies $v(f(a)-f(b)), v(g(a)-g(b)) > \varepsilon + \gamma$.
	
	Now, we see that if $v(a-b) > \delta$, then we have $v(\varphi(b) - \varphi(a)) \geq \min\left\{v\left(\frac{f(b)-f(a)}{g(a)}\right), v\left(\frac{f(b)}{g(b)}\cdot \frac{g(b)-g(a)}{g(a)}\right)\right\}$. We know that $v\left(\frac{f(b)-f(a)}{g(a)}\right) > \varepsilon + \gamma - \gamma = \varepsilon$ and $v\left(\frac{f(b)}{g(b)}\cdot \frac{g(b)-g(a)}{g(a)}\right) = v\left(\frac{f(b)}{g(b)}\right) + v\left(\frac{g(b)-g(a)}{g(a)}\right) > 0 + \varepsilon + \gamma - \gamma = \varepsilon$ since $\frac{f(b)}{g(b)} \in D \subseteq V$. Thus, $v(\varphi(b) - \varphi(a)) > \varepsilon$, showing that $\varphi$ is continuous at $a$, which means that $\varphi$ is continuous since $a \in E$ was arbitrarily chosen. 
\end{proof}

For studying rings of integer-valued rational functions, there are certain Prüfer domains that are of particular interest to us. The following definitions are from \cite{IntValuedRational}:
\begin{definition}
	A Prüfer domain $D$ is \textbf{monic} if there is a monic unit-valued polynomial in $D$. 
\end{definition}

\begin{definition}
	Let $D$ be a Prüfer domain. The Prüfer domain $D$ is \textbf{singular} if there exists a family $\Lambda$ of maximal ideals of $D$ such that 
	\begin{itemize}
		\item $D = \bigcap\limits_{\m \in \Lambda} D_\m$,
		\item for each $\m \in \Lambda$, the maximal ideal of $D_\m$ is a principal ideal, generated by some $t_\m \in D_\m$, and
		\item there is an element $t \in D$ and an integer $n$ such that, for each $\m \in \Lambda$, $0< v_\m(t) < n v_\m(t_\m)$, where $v_\m$ is the valuation associated with $D_\m$. 
	\end{itemize}
\end{definition}

An important invariant of a Prüfer domain $D$ is the \textbf{Picard group} of $D$, denoted as $\Pic(D)$. The Picard group is defined to be the set of finitely generated ideals of $D$ modulo the principal ideals of $D$. The group operation is ideal multiplication.

The following result is stated only for $\IntR(D)$ in \cite{PruferNonDRings}, but the same proof can be used to get the same statement about $\IntR(E,D)$.

\begin{theorem}\cite{PruferNonDRings}
	Let $D$ be a monic Prüfer domain with $E$ a subset of $K$, the quotient field of $D$. Then $\IntR(E,D)$ is a Prüfer domain with torsion Picard group.
\end{theorem}

 For a Prüfer domain $D,$ the Picard group can be seen as a way to measure how far $D$ is from being a Bézout domain, since a Bézout domain has trivial Picard group. This next result shows that rings of integer-valued rational functions over singular Prüfer domains are Bézout domains, which implies that they are Prüfer domains as well.

\begin{theorem}\cite{IntValuedRational}
	Let $D$ be a singular Prüfer domain with $E$ a subset of $K$, the quotient field of $D$. Then $\IntR(E,D)$ is a Bézout domain.
\end{theorem}

We take the results of \cite{PruferNonDRings, IntValuedRational} and restrict ourselves to looking at rings of integer-vauled rational functions over valuation rings. 

\begin{remark}
	Note that a valuation domain is monic if and only if its residue field is not algebraically closed, and a valuation domain is singular if and only if its maximal ideal is principal.
\end{remark}

\begin{corollary}
	If $V$ is a valuation domain with a residue field that is not algebraically closed or a principal maximal ideal, then $\IntR(V)$ is a Prüfer domain.
\end{corollary}

In this section, we will explore the converse. To that end, we need a few definitions.

\begin{definition}
	Let $\Gamma$ be an abelian group with a total order on its elements. We say that $\Gamma$ is a \textbf{totally ordered abelian group} if for $\alpha, \beta, \gamma, \delta \in \Gamma$, $\alpha\leq \beta$ and $\gamma\leq \delta$ imply that $\alpha+\gamma \leq \beta+\delta$.

\end{definition}

\begin{definition}
	A totally ordered abelian group $\Gamma$ is \textbf{divisible} if for all $\gamma \in \Gamma$ and nonzero $n \in \Z$, there exists $\delta \in \Gamma$ such that $n\delta = \gamma$. 
\end{definition}

\begin{definition}
	Let $\Gamma$ be a totally ordered abelian group. We can define its \textbf{divisible hull} \[\Q\Gamma = \left\{ \frac{\gamma}{n} \mid \gamma \in \Gamma, n \in \Z_{> 0} \right\}/\sim,\] where $\frac{\gamma}{n} \sim \frac{\gamma'}{m}$ if $m\gamma = n\gamma'$ in $\Gamma$. Then define the group operation to be $\frac{\gamma}{n} + \frac{\gamma'}{m} = \frac{m\gamma + n\gamma'}{nm}$ and the ordering to be $\frac{\gamma}{n} \leq \frac{\gamma'}{m}$ if and only if $m\gamma \leq n\gamma'$. Equivalently, we may define  $\Q\Gamma = \Gamma \otimes_\Z \Q$.

	Furthermore, if $v:K \to \Gamma \cup \{\infty \}$ is a valuation, we can extend the valuation to $\tilde{v}: K(t_\gamma \mid \gamma \in \Q\Gamma) \to \Q\Gamma\cup \{\infty \}$ defined as the monomial extension mapping $t_\gamma$ to $\gamma$ for every $\gamma \in \Q\Gamma$. More explicitly, for nonzero elements of $K[t_\gamma \mid \gamma \in \Q\Gamma]$, we define $\tilde{v}$ to be
	\[
		\tilde{v}\left( \sum\limits_{i_1, \dots, i_n} a_{i_1\cdots i_n} t_{\gamma_1}^{c_{1,i_1\cdots i_n}} \cdots t_{\gamma_n}^{c_{n, i_1\cdots i_n}}  \right) = \min\left\{ v(a_{i_1\dots i_n}) + \sum_{k=1}^n c_{k, i_1\cdots i_n} \gamma_k \right\},
	\]
	where each $a_{i_1\cdots i_n} \in K$. Then this extends uniquely to $K(t_\gamma \mid \gamma \in \Q\Gamma)$.

\end{definition}

	\begin{remark}
	The divisible hull $\Q\Gamma$ of $\Gamma$ is divisible and extends the ordering on $\Gamma$.
\end{remark}

\begin{remark}
	If $\m$ is principal, then the value group $\Gamma$ is not divisible. Say $\m = (\varpi)$ for some $\varpi \in V$. There does not exist an element $a \in V$ such that $2v(a) = v(\varpi)$. 
\end{remark}

In order to consider many valuations at once, we introduce the notions of minimum valuation functions and local polynomials. The minimal valuation function is closely related to monomial valuations. 

\begin{definition}
	Take a nonzero polynomial $f \in K[x]$ and write it as $f(x) = a_nx^n + \cdots + a_1x+ a_0$ for $a_0, a_1, \dots, a_n \in K$. We define the \textbf{minimum valuation function of $f$} as $\minval_{f,v}: \Gamma \to \Gamma $ by \[\gamma \mapsto \min\{v(a_0), v(a_1)+\gamma, v(a_2)+2\gamma, \dots, v(a_n) + n\gamma\}\] for each $\gamma \in \Gamma$. We will denote $\minval_{f,v}$ as $\minval_f$ if the valuation $v$ is clear from context. It is oftentimes helpful to think of $\minval_f$ as a function from $\Q{\Gamma}$ to $\Q{\Gamma}$ defined as $\gamma \mapsto \min\{v(a_0), v(a_1)+\gamma, v(a_2)+2\gamma, \dots, v(a_n) + n\gamma\}$ for each $\gamma \in \Q{\Gamma}$.

	In the same setup, taking $t \in K$, we can define the \textbf{local polynomial of $f$ at $t$} to be \[\loc_{f, v, t}(x) = \frac{f(tx)}{a_dt^d} \mod \m,\] where  $d = \max\{i \in \{0, 1, \dots, n\} \mid v(a_i) + iv(t) = \minval_f(v(t)) \}$. Again, we may omit the valuation $v$ in $\loc_{f, v, t}(x)$ and write $\loc_{f, t}(x)$ if the valuation is clear from the context. A priori, we do not know if the coefficients of $\frac{f(tx)}{a_dt^d}$ are in $V$. We need to confirm this so that the local polynomial is well-defined.
	
\end{definition}

\begin{remark}
	Using the notation above, we compare the minimum valuation function to monomial valuations. In Definition 3.3 of \cite{Peruginelli}, the monomial valuation (centered at 0) $v_{0, \gamma}$ is defined to be $v_{0, \gamma}(f) = \min\{v(a_i) + i\gamma \mid i = 0, \dots, n\}$. This shows that $\minval_f(\gamma) = v_{0, \gamma}(f)$, so we can think of the minimum valuation function as ranging over various monomial valuations. 
\end{remark}

\begin{proposition}\label{Prop:LocalPolys}
	Let $f \in K[x]$ be some nonzero polynomial and $t \in K$. Write $f(x) = a_nx^n + \cdots + a_1x+ a_0$ for $a_0, a_1, \dots, a_n \in K$. Then the local polynomial of $f$ at $t$ is a well-defined monic polynomial in $V/\m[x]$ of degree $d$, where
	\[
	d \coloneqq\max\{i \in \{0, 1, \dots, n\} \mid v(a_it^i) = v(a_i) + iv(t) = \minval_f(v(t)) \}.
	\]
	Moreover, if $i_1 < \cdots < i_s$ are the indices $i$ such that $\minval_f(v(t)) = v(a_i) + iv(t)$, then
	\[
	\loc_{f, t}(x) = r_{i_1} x^{i_1} + \cdots + r_{i_{s-1}}x^{i_{s-1}} + x^{i_s}, 
	\]
	for some nonzero elements $r_{i_1}, \dots, r_{i_{s-1}} \in V/\m$. 
\end{proposition}

\begin{proof}
	We have that $f(tx) = a_0 + a_1tx + a_1t^2 x^2 + \cdots + a_n t^n x^n$. Then consider each coefficient of $\frac{f(tx)}{a_dt^d}$. We calculate that $v\left(\frac{a_it^i}{a_d t^d} \right) \geq 0$ for all $i$ since $\minval_f(v(t)) \leq v(a_it^i)$ for all $i$. This shows that $\frac{f(tx)}{a_dt^d} \in V[x]$. More specifically, $v\left(\frac{a_it^i}{a_d t^d} \right) = 0$ if and only if $v(a_it^i)  = \minval_f(v(t))$. Therefore, for $i > d$, we have $v\left(\frac{a_it^i}{a_d t^d} \right) > 0$. Moreover, the coefficient of the degree $d$ term of $\frac{f(tx)}{a_dt^d}$ is 1, so $\frac{f(tx)}{a_dt^d} \mod \m$ has degree $d$ with leading coefficient 1. 
	
	Additionally, since $v\left(\frac{a_it^i}{a_d t^d} \right) = 0$ if and only if $v(a_it^i)  = \minval_f(v(t))$, we know that $x^i$ has a nonzero coefficient in $\loc_{f, t}(x)$ exactly when $\minval_f(v(t)) = v(a_i) + iv(t)$.
\end{proof}

Now we establish the minimal valuation function as a piecewise linear function. Furthermore, we deduce that the slopes of the minimal valuation function can be obtained from the highest and lowest degree terms in certain local polynomials. 

\begin{proposition}\label{Prop:FormOfMinval}
	For a nonzero $f \in K[x]$, the function $\minval_f$ has the following form evaluated at $\gamma \in \Q\Gamma$
	\[
	\minval_f(\gamma) = \begin{cases}
		c_1 \gamma + \beta_1, & \gamma \leq \delta_1,\\
		c_2 \gamma + \beta_2, & \delta_1 \leq \gamma \leq \delta_2,\\
		\vdots\\
		c_{k-1} \gamma + \beta_{k-1}, & \delta_{k-2} \leq \gamma  \leq \delta_{k-1},\\
		c_k \gamma + \beta_k, & \delta_{k-1} \leq \gamma,
	\end{cases}
	\] 
	where $c_1, \dots, c_k \in \N$ such that $c_1 > \cdots > c_k$; $\beta_1, \dots, \beta_k \in \Gamma$; and $\delta_1, \dots, \delta_{k-1} \in \Q{\Gamma}$ such that $\delta_1 < \cdots < \delta_{k-1}$. 
	
	Moreover, suppose that $t \in K$ is such that $v(t) = \delta_i$ for some $i \in \{1, \dots, k-1\}$. Write $\loc_{f, t}(x) = r_{j_{i,t,1}}x^{j_{i,1}} + \cdots + r_{j_{i,t,s_i-1}}x^{j_{i,s_i-1}} + x^{j_{i,s_i}}$ with $j_{i,1} < \cdots < j_{i,s_i}$, where $r_{j_{i,t,1}}, \dots, r_{j_{i,t,s_i-1}} \in V/\m$ and are all nonzero. Then $c_{i+1} = j_{i, 1}$ and $c_i = j_{i, s_i}$. 
\end{proposition}

\begin{proof}

	Write $f(x) = a_nx^n + \cdots + a_1x+ a_0$, where $a_0, a_1, \dots, a_n \in K$. We know that $\minval_f(\gamma) = \min\{j\gamma + v(a_j) \mid j \in \{0, \dots, n\}\}$ for all $\gamma \in \Q\Gamma$. Then let $\delta_1 < \cdots < \delta_{k-1}$ be the elements of $\Q{\Gamma}$ such that $\minval_f(\delta_i) = j\delta_i + v(a_j) $ for at least two indices $j$. Since $j\gamma + v(a_j) =  j'\gamma + v(a_{j'})$ for $j \neq j'$ if and only if $\gamma = \frac{v(a_{j'})-v(a_j)}{j-j'}$, we know that $\{\delta_1, \dots, \delta_{k-1}\}$ is a subset of the finite set $\left\{\frac{v(a_{j'})-v(a_j)}{j-j'} \,\middle\vert\, j,j' \in \{0, \dots, n\}, j \neq j' \right\}$.  For convenience, set $\delta_0 = - \infty$ and $\delta_k = \infty$. 
	
	For $i \in \{1,\dots, k-1\}$, we set \[c_i \coloneqq \max\{ j \in \{0,1, \dots, n \} \mid \minval_f(\delta_i) = j\delta_i + v(a_j)  \}.\] Also set \[c_k \coloneqq \min\{j \in \{0,1, \dots, n\} \mid \minval_f(\delta_{k-1}) = j\delta_{k-1} + v(a_j) \}.\] Next, set $\beta_i \coloneqq v(a_{c_i})$ for $i \in \{1, \dots, k\}$. 
	
	Fix $i$ in $\{1, \dots, k-1\}$. We want to show $\minval_f(\gamma) = c_i\gamma + \beta_i$ for all $\gamma \in \Q\Gamma$ such that $\delta_{i-1} \leq \gamma \leq \delta_i$. Suppose not. Then there exists $\gamma \in \Q\Gamma$ such that $\delta_{i-1} \leq \gamma < \delta_i$ and $\minval_f(\gamma) < c_i\gamma + \beta_i$. We now have that $\minval_f(\gamma) = j\gamma + v(a_j) < c_i\gamma + \beta_i$ for some $j \in \{0, \dots, n\}$ different from $c_i$. Thus, $(j-c_i)\gamma < \beta_i - v(a_j)$. Also note that $j\delta_i + v(a_j) \geq c_{i}\delta_i + \beta_i$, so $(j-c_i)\delta_i \geq \beta_i - v(a_j)$. These inequalities imply that $(j-c_i)\gamma < \beta_i - v(a_j) \leq (j-c_i)\delta_i$. Because $\gamma < \delta_i$, we can deduce that $j > c_i$. Using the inequalities again shows that $\gamma < \frac{\beta_i-v(a_j)}{j-c_i} \leq \delta_i$. We want both inequalities to be strict. If $\frac{\beta_i-v(a_j)}{j-c_i} = \delta_i$, then $j\delta_i + v(a_j) = c_i\delta_i + \beta_i$. The fact that $j > c_i$ contradicts the maximality of $c_i$. Thus, $\gamma < \frac{\beta_i-v(a_j)}{j-c_i} < \delta_i$. Due to the fact that $j\delta + v(a_j) = c_i\delta + \beta_i$ for $\delta = \frac{\beta_i-v(a_j)}{j-c_i}$ and the way $\delta_1, \dots, \delta_{k-1}$ are picked out, we know that $\minval_f\left(\frac{\beta_i-v(a_j)}{j-c_i}\right) < c_i \cdot \frac{\beta_i-v(a_j)}{j-c_i} + \beta_i$. We replace $\gamma$ with $\frac{\beta_i-v(a_j)}{j-c_i}$ and repeat the argument. Since $\delta_{i-1} \leq \gamma < \frac{\beta_i-v(a_j)}{j-c_i} < \delta_i$ and there are only finitely many elements of the form $\frac{\beta_i-v(a_j)}{j-c_i}$ for $j \neq c_i$, this argument cannot be repeated infinitely. Therefore, a contradiction is reached eventually after a finite number of repetitions. This shows that $\minval_f(\gamma) = c_i\gamma + \beta_i$ for all $\gamma \in \Q\Gamma$ such that $\delta_{i-1} \leq \gamma \leq \delta_i$. A similar argument will show that $\minval_f(\gamma) = c_k\gamma + \beta_k$ for all $\gamma \in \Q\Gamma$ such that $\gamma \geq \delta_{k-1}$. 
	
	Again, fix $i \in \{1, \dots, k-1\}$. Suppose that $t \in K$ is any element such that $v(t) = \delta_i$. Write $\loc_{f, t}(x) = r_{j_{i,t,1}}x^{j_{i,1}} + \cdots + r_{j_{i,t,s_i-1}}x^{j_{i,s_i-1}} + x^{j_{i,s_i}}$ with $j_{i,1} < \cdots < j_{i,s_i}$, where $r_{j_{i,t,1}}, \dots, r_{j_{i,t,s_i-1}} \in V/\m$ and are all nonzero and each $s_i \geq 2$. The local polynomial has this form due to Proposition \ref{Prop:LocalPolys}. We also furthermore know that $j_{i,1} < \cdots < j_{i,s_i}$ are all the indices $j$ such that $\minval_f(\delta_i) =  j\delta_i + v(a_j)$. We then have $c_i = j_{i,s_i}$ since $c_i$ is the maximum of all such indices. If $i = k-1$, we have $c_k = j_{i,1}$ by the definition of $c_k$. Now suppose that $i < k-1$ and $c_{i+1} > j_{i,1}$ for a contradiction. The fact that $j_{i,1}\delta_i + \beta_{j_{i,1}} = c_{i+1}\delta_i + \beta_{i+1}$ implies
	\begin{align*}
		j_{i,1}\delta_{i+1} + \beta_{j_{i,1}} &=j_{i,1}\delta_i + \beta_{j_{i,1}}  + j_{i,1}(\delta_{i+1}-\delta_i)\\
		& = c_{i+1}\delta_i +\beta_{i+1} + j_{i,1}(\delta_{i+1}-\delta_i) \\
		& <c_{i+1}\delta_i +\beta_{i+1} + c_{i+1}(\delta_{i+1}-\delta_i)\\
		& = c_{i+1}\delta_{i+1} + \beta_{i+1} \\
		& = \minval_f(\delta_{i+1}),
	\end{align*}
	contradicting the definition of $\minval_f(\delta_{i+1})$. This shows that $c_{i+1} = j_{i,1}$. Note also that $s_i > 1$ by the definition of $\delta_i$. Thus, $c_{i+1} = j_{i,1} < j_{i,s_i} = c_i$, which implies that $c_1 > \cdots > c_k$.
\end{proof}

\begin{remark}
	We see that if the value group $\Gamma$ is divisible, then $\delta_1, \dots, \delta_k$ given above are always in $\Gamma$. 
\end{remark}

Now we give some results about how information about the valuation of the polynomial evaluations can be extracted from the minimum valuation polynomial and the local polynomials. 

\begin{lemma}\label{Lem:MinvalLowerBound}
	Let $f \in K[x]$ be a nonzero polynomial. For all nonzero $t \in K$, we have $\minval_f(v(t)) \leq v(f(t))$. 
\end{lemma}

\begin{proof}
	Write $f(x) = a_nx^n+ \cdots +a_1x+a_0$, where $a_0, \dots, a_n \in K$. Then 
	\[
	v(f(t)) = v\left(\sum_{i=0}^{n}a_i t^i \right) \geq \min\{v(a_it^i) \mid i \in \{0, \dots, n\} \} = \minval_f(v(t)). 
	\]
\end{proof}

In other words, the minimum valuation function serves as a lower bound for the valuation of polynomial evaluations. We now characterize when this lower bound is strict or not using a local polynomial.

\begin{proposition}\label{Prop:LocalPolysRoots}
	Let $f \in K[x]$ be nonzero and $t \in K$. Then there exists an $s \in K$ with $v(s) = v(t)$ such that $v(f(s))> \minval_f(v(t))$ if and only if $\loc_{f, t}(x)$ has a nonzero root. More specifically, for $u \in V$ such that $v(u) = 0$, we have that $v(f(tu)) > \minval_f(v(t))$ if $\loc_{f, t}(u + \m) = 0$ and $v(f(tu)) = \minval_f(v(t))$ if $\loc_{f, t}(u + \m) \neq 0$. 
\end{proposition}

\begin{proof}
	Write $f(x) = a_0 + a_1x + \cdots + a_nx^n$ with each $a_i \in K$. Let $i_1 < \cdots < i_r$ be all the indices $i$ such that $\minval_f(v(t)) = iv(t) + v(a_i)$. Then \[\frac{f(tx)}{a_{i_r}t^{i_r}} = \frac{a_0}{a_{i_r}t^{i_r}} + \frac{a_1t}{a_{i_r}t^{i_r}}x + \frac{a_2t^2}{a_{i_r}t^{i_r}}x^2 + \cdots \frac{a_nt^n}{a_{i_r}t^{i_r}}x^n. \]
	Taking this modulo $\m$, we get
	\[
	\loc_{f, t}(x) = \frac{a_{i_1}t^{i_1}}{a_{i_r}t^{i_r}}x^{i_1} + \cdots + \frac{a_{i_{r-1}}t^{i_{r-1}}}{a_{i_r}t^{i_r}}x^{i_{r-1}} + x^{i_r} \mod \m. 
	\]
	Take $s \in K$ with $v(s) = v(t)$. Set $u = \frac{s}{t}$. Note that $u \neq 0 \mod \m$. Then $v(f(s)) > \minval_f(v(t)) = v(a_{i_r}t^{i_r})$ if and only if
	\[
	\frac{f(tu)}{a_{i_r}t^{i_r}} = \frac{a_0}{a_{i_r}t^{i_r}} + \frac{a_1t}{a_{i_r}t^{i_r}}u + \frac{a_2t^2}{a_{i_r}t^{i_r}}u^2 + \cdots \frac{a_nt^n}{a_{i_r}t^{i_r}}u^n
	\]
	has valuation under $v$ that is strictly greater than 0, which can happen if and only if 
	\[
	\loc_{f,t}(u +\m) =  \frac{a_{i_1}t^{i_1}}{a_{i_r}t^{i_r}}u^{i_1} + \cdots + \frac{a_{i_{r-1}}t^{i_{r-1}}}{a_{i_r}t^{i_r}}u^{i_{r-1}} + u^{i_r} \mod \m = 0. 
	\]
\end{proof}

We observe that most of the time, the local polynomial is a monomial, which does not have nonzero roots. Thus, the minimum valuation function gives the valuation of the polynomial evaluations most of the time.

\begin{corollary}\label{Cor:MinvalAlmostEverywhere}
 Take $f \in K[x]$ to be a nonzero polynomial. Then using the notation of Proposition \ref{Prop:FormOfMinval}, if $t \in K$ is such that $v(t) \neq \delta_i$ for all $i \in \{1, \dots, k-1\}$, then $v(f(t)) = \minval_f(v(t))$. 
\end{corollary}

\begin{proof}
	Write $f(x) = a_nx^n + \cdots + a_1x+ a_0$, where $a_0, \dots, a_n \in K$. Since $v(t) \neq \delta_i$ for any $i$, we have that $\minval_f(v(t)) = jv(t) + v(a_j)$ for a unique $j$ in $\{0, \dots, n\}$ by the definition of $\delta_1, \dots, \delta_{k-1}$ in Proposition \ref{Prop:FormOfMinval}. Thus, $\loc_{f,t}(x) = x^j$, which has no nonzero roots, so $v(f(t)) = \minval_f(v(t))$ by Proposition \ref{Prop:LocalPolysRoots}.
\end{proof}

If we further assume that $V$ is a valuation domain with infinite residue field, we can say something about the finitely many exceptions $\delta_1, \dots, \delta_{k-1}$ of the value group. In this case, every element of the value group has some element of that value whose evaluation has valuation determined by the minimum valuation function.

\begin{corollary}\label{Cor:AttainMinval}
	Suppose $V$ has infinite residue field. If $f_1, \dots, f_m \in K[x]$ are a finite number of nonzero polynomials, then for any $\gamma \in \Gamma$, there exists an $a \in K$ with $v(a) = \gamma$ such that $v(f_i(a)) = \minval_{f_i}(\gamma)$ for each $i = 1,\dots, m$. 
\end{corollary}

\begin{proof}
	Fix $\gamma \in \Gamma$ and $t \in K$ with $v(t) = \gamma$. Let $u \in K$ with $v(u) = 0$. We know for any nonzero polynomial $f \in K[x]$ that $v(f(tu))>\minval_f(v(t))$ if and only if $\loc_{f, t}(u+\m) = 0$. However, $\loc_{f_1, t}, \dots, \loc_{f_m, t}$ can collectively only have a finite number of nonzero roots in $V/\m$. Thus, there exists $u \in V^\times$ such that $\loc_{f_i, t}(u+\m) \neq 0$ for all $i$ and hence $v(f_i(tu)) = \minval_{f_i}(\gamma)$ for each $i = 1,\dots, m$.

\end{proof}

\begin{proposition}
	Let $f, g \in K[x]$ be nonzero polynomials. Then \[\minval_{fg} = \minval_f + \minval_g.\]
\end{proposition}

\begin{proof}
	We will view $f$ and $g$ as being in $K(t_\gamma\mid \gamma \in \Q\Gamma)[x]$ and take the monomial valuation $v$ that extends $V$ mapping $t_\gamma$ to $\gamma$ for each $\gamma \in \Q\Gamma$. 
	
	By Corollary \ref{Cor:MinvalAlmostEverywhere}, we know that for all but finitely many $\gamma \in \Q\Gamma$, we have that $v(f(a)) = \minval_f(\gamma)$, $v(g(a)) = \minval_g(\gamma)$, and $v((fg)(a)) = \minval_{fg}(\gamma)$ for all $a \in K(t_\gamma\mid \gamma \in \Q\Gamma)$ such that $v(a) = \gamma$. 
	For such values of $\gamma$, we take $a_\gamma \in K(t_\gamma\mid \gamma \in \Q\Gamma)$ such that $v(a_\gamma) = \gamma$. This means that
	\[
	\minval_{fg}(\gamma) = v((fg)(a_\gamma)) = v(f(a_\gamma)) + v(g(a_\gamma)) = \minval_f(\gamma) + \minval_g(\gamma)
	\]
	holds for all but finitely many $\gamma \in \Q\Gamma$. Since $\minval_{fg}, \minval_f$, and $\minval_g$ are all piecewise linear functions from $\Q\Gamma$ to $\Q\Gamma$ by Proposition \ref{Prop:FormOfMinval}, we have equality for all $\gamma \in \Q\Gamma$ and in particular for all $\gamma \in \Gamma$. 
\end{proof}

The previous proposition also follows using the fact that $\minval_{fg}(\gamma) = v_{0, \gamma}(fg) = v_{0, \gamma}(f) + v_{0, \gamma}(g) = \minval_f(\gamma) + \minval_g(\gamma)$ for every $\gamma \in \Q\Gamma$. 

We now use this fact to define the minimum valuation function of a nonzero rational function.

\begin{definition}
	Let $\varphi \in K(x)$ be a nonzero rational function. Write $\varphi = \frac{f}{g}$ for some $f, g \in K[x]$. For $\gamma \in \Gamma$, we define $\minval_\varphi(\gamma) = \minval_f(\gamma) - \minval_g(\gamma)$, the \textbf{minimum valuation function of $\varphi$}. 
	
	This is well defined. If $\frac{f}{g} = \frac{F}{G}$ for some $F, G \in K[x]$, then $fG = gF$ and then $\minval_f + \minval_G = \minval_g + \minval_F$, which means that $\minval_f - \minval_g = \minval_F-\minval_G$. 
\end{definition}

There is an analog for Proposition \ref{Prop:FormOfMinval} giving the form of the minimum valuation function of a rational function. Note that the ordering on the coefficients of $\gamma$ is lost and these coefficients can be negative.

\begin{proposition}\label{Prop:FormOfMinvalRat}
	 For a nonzero $\varphi \in K(x)$, the function $\minval_\varphi$ has the following form evaluated at $\gamma \in \Q\Gamma$
	\[
	\minval_\varphi(\gamma) = \begin{cases}
		c_1 \gamma + \beta_1, & \gamma \leq \delta_1,\\
		c_2 \gamma + \beta_2, & \delta_1 \leq \gamma \leq \delta_2,\\
		\vdots\\
		c_{k-1} \gamma + \beta_{k-1}, & \delta_{k-2} \leq \gamma  \leq \delta_{k-1},\\
		c_k \gamma + \beta_k, & \delta_{k-1} \leq \gamma,
	\end{cases}
	\] 
	where $c_1, \dots, c_k \in \Z$; $\beta_1, \dots, \beta_k \in \Gamma$; and $\delta_1, \dots, \delta_{k-1} \in \Q{\Gamma}$ such that $\delta_1 < \cdots < \delta_{k-1}$. 
\end{proposition}

\begin{proof}
	Write $\varphi = \frac{f}{g}$ for some polynomials $f, g \in K[x]$. By Proposition \ref{Prop:FormOfMinval}, we know there are $\delta_1, \dots, \delta_{k-1}, \delta_1', \dots, \delta_{k'-1}'$ such that $-\infty = \delta_0 < \delta_1 < \cdots < \delta_{k-1} < \delta_k = \infty$ and $-\infty = \delta_0' < \delta_1' < \cdots < \delta_{k'-1}' < \delta_{k'} = \infty$ and that for all $\gamma \in \Q\Gamma$ such that $\gamma$ is between $\delta_i$ and $\delta_{i+1}$, we have some $c_i \in \N$ and $\beta_i \in \Gamma$ such that
	\[\minval_f(\gamma) = c_i\gamma + \beta_i\]
	and for all $\gamma \in \Q\Gamma$ such that $\gamma$ is between $\delta_i'$ and $\delta_{i+1}'$, we have some $c_i' \in \N$ and $\beta_i' \in \Gamma$ such that
	
	\[
	\minval_g(\gamma) = c_i'\gamma + \beta_i'.
	\]
	Now order the elements of the set $\{\delta_0, \dots, \delta_k, \delta_0', \dots, \delta_k'\}$ and rename the elements $\delta_0'', \dots, \delta_{k''}''$ so that $-\infty = \delta_0'' < \delta_1'' < \cdots < \delta_{k''-1}'' < \delta_{k''}'' = \infty$. Let $r$ be such that $r \in \{0, \dots, k''-1\}$. We know that the interval between $\delta_r''$ and $\delta_{r+1}''$ is contained in the interval between $\delta_{i}$ and $\delta_{i+1}$ and also contained in the interval between $\delta_j'$ and $\delta_{j+1}'$ for some $i$ and $j$. Thus, for all $\gamma \in \Q\Gamma$ such that $\delta_r'' \leq \gamma \leq \delta_{r+1}''$, we have
	\[
	\minval_\varphi(\gamma) = \minval_f(\gamma) - \minval_g(\gamma) = (c_i - c_j')\gamma + (\beta_i-\beta_j'),
	\]
	giving us the desired form for $\minval_\varphi$.
\end{proof}

Even though the ordering on the coefficients of $\gamma$ in $\minval_\varphi$ is lost, the coefficients can still give information about the powers that appear in the local polynomials.

\begin{lemma}\label{Lem:PowersOfLoc}
	Take $\varphi \in K(x)$ to be nonzero and $\alpha \in \Gamma$. There exist $\varepsilon \in \Q\Gamma$ with $\varepsilon > 0$ small enough, $c, c' \in \Z$, and $\beta, \beta' \in \Gamma$ such that
	\[
	\minval_\varphi(\gamma) = \begin{cases}
		c\gamma + \beta, & \text{if $\alpha - \varepsilon < \gamma < \alpha$},\\
		c'\gamma + \beta', & \text{if $\alpha < \gamma < \alpha + \varepsilon$}. 
	\end{cases}
	\]
	Write $\varphi = \frac{f}{g}$ for some $f, g \in K[x]$. Take $t \in K$ such that $v(t) = \gamma$. We can write $\loc_{f, t} = a_{i_1}x^{i_1} + \cdots + a_{i_r}x^{i_r}$ and $\loc_{g, t} = b_{j_1}x^{j_1}+ \cdots + b_{j_s}x^{j_s}$ for some nonzero $a_{i_1}, \dots, a_{i_r}, b_{j_1}, \dots, b_{j_s} \in V/\m$. Then
	\[
	c = i_r - j_s \quad \text{and} \quad c' = i_1-j_1.
	\]
\end{lemma}

\begin{proof}
	Due to Proposition \ref{Prop:FormOfMinval}, we can make $\varepsilon$ small enough so that there exist $\zeta,\zeta',\eta,\eta' \in \Gamma$ such that 
	\[
	\minval_f(\gamma) = \begin{cases}
		i_r\gamma + \zeta, & \text{if $\alpha - \varepsilon < \gamma < \alpha$},\\
		i_1\gamma + \zeta', & \text{if $\alpha < \gamma < \alpha + \varepsilon$}
	\end{cases}
	\]
	and
	\[
	\minval_g(\gamma) = \begin{cases}
		j_s\gamma + \eta, & \text{if $\alpha - \varepsilon < \gamma < \alpha$},\\
		j_1\gamma + \eta', & \text{if $\alpha < \gamma < \alpha + \varepsilon$}. 
	\end{cases}
	\]
	Since $\minval_\varphi = \minval_f - \minval_g$, we obtain $c = i_r - j_s$ and $c' = i_1 - j_1$. 
\end{proof}

However, for rational functions, there is no analog of Lemma \ref{Lem:MinvalLowerBound}, since it is possible that $\minval_\varphi(v(t))$ is greater than, equal to, or less than $v(\varphi(t))$. To calculate how $\minval_\varphi(v(t))$ compares with $v(\varphi(t))$, we can try to apply Proposition \ref{Prop:LocalPolysRoots} to the local polynomials of the numerator and the denominator of $\varphi$. This does not give a definite answer in the case when the local polynomials have a common nonzero root, so further calculations are needed in this case. Nevertheless, there are analogs of Corollary \ref{Cor:MinvalAlmostEverywhere} and Corollary \ref{Cor:AttainMinval} that say $\minval_\varphi(v(t))$ and $v(\varphi(t))$ are equal most of the time.

\begin{lemma}\label{Lem:MinvalEqualityAE}
	Take $\varphi \in K(x)$ to be a nonzero rational function. For all but finitely many $\gamma \in \Gamma$, we have that $v(\varphi(t)) = \minval_\varphi(v(t))$ for all $t \in K$ such that $v(t) = \gamma$. 
\end{lemma}

\begin{proof}
	We write $\varphi = \frac{f}{g}$ for some $f, g \in K[x]$. Corollary \ref{Cor:MinvalAlmostEverywhere} tells us that for all but finitely many $\gamma \in \Gamma$, we have $v(f(t)) = \minval_f(v(t))$ for all $t \in K$ such that $v(t) = \gamma$. Similarly, for all but finitely many $\gamma \in \Gamma$, we have $v(g(t)) = \minval_g(v(t))$ for all $t \in K$ such that $v(t) = \gamma$. There are still only finitely many values of $\gamma \in \Gamma$ we need to exclude. Thus, for all but finitely many $\gamma \in \Gamma$, we have that \[v(\varphi(t)) = v(f(t)) - v(g(t)) = \minval_f(v(t)) - \minval_g(v(t)) =  \minval_\varphi(v(t))\] for all $t \in K$ such that $v(t) = \gamma$. 
\end{proof}

\begin{proposition}\label{Prop:AttainMinvalRat}
	Suppose the residue field of $V$ is infinite. If we have nonzero $\varphi_1, \dots, \varphi_m \in K(x)$, then for any $\gamma \in \Gamma$, there exists $a \in K$ with $v(a) = \gamma$ such that $\minval_{\varphi_i}(\gamma) = v(\varphi_i(a))$ for all $i$. 
\end{proposition}

\begin{proof}
	Write $\varphi_i = \frac{f_i}{g_i}$ with $f_i,g_i \in K[x]$ for all $i$. By applying Corollary \ref{Cor:AttainMinval} to $f_1,\dots, f_m, g_1, \dots, g_m$, we see that for any $\gamma \in \Gamma$, there exists $a\in K$ with $v(a) = \gamma$ such that $\minval_{f_i}(\gamma) = v(f_i(a))$ and $\minval_{g_i}(\gamma) = v(g_i(a))$ for all $i$. Putting these together yields
	\[
	\minval_{\varphi_i}(\gamma) = \minval_{f_i}(\gamma) - \minval_{g_i}(\gamma) = v(f_i(a)) - v(g_i(a)) = v(\varphi_i(a))
	\]
	for all $i$.	
\end{proof}

The following lemma will be useful in describing the case of a valuation domain with algebraically closed residue field and maximal ideal that is not principal. 
\begin{lemma}\label{Lem:ValuationLemma}
	Suppose that $V/\m$ is algebraically closed and $\m$ is not a principal ideal of $V$. Let $\varphi \in K(x)$ be a nonzero rational function such that there exist $\alpha, \epsilon \in \Gamma$ with $\epsilon > 0$ so that
	\[
	\minval_\varphi(\gamma) = \begin{cases}
		c_1\gamma + \beta_1, & \text{ if $\alpha - \epsilon \leq \gamma \leq \alpha$,}\\
		c_2\gamma + \beta_2, & \text{ if $\alpha \leq \gamma \leq \alpha + \epsilon$,}
	\end{cases}
	\]
	for some $c_1, c_2 \in \Z$ and $\beta_1, \beta_2 \in \Gamma$. 
	
	If $c_1 > c_2$, then there exists $a \in K$ with $v(a) = \alpha$ and $v(\varphi(a)) > \minval_\varphi(\alpha)$. 
	
	If $c_1 < c_2$, then there exists $a \in K$ with $v(a) = \alpha$ and $v(\varphi(a)) < \minval_\varphi(\alpha)$. 
\end{lemma}

\begin{proof}
	Write $\varphi = \frac{f}{g}$ for some $f, g \in K[x]$. Let $b$ be an element of $K$ such that $v(b) = \alpha$. Then we can completely factor $\loc_{f,b}(x)$ and $\loc_{g, b}(x)$ as $V/\m$ is algebraically closed, so those polynomials have the forms
	\[
	\loc_{f,b}(x) = x^i(x- \xi_1)^{e_1} \cdots (x- \xi_n)^{e_n} \quad \text{and} \quad \loc_{g,b}(x) = x^j(x- \xi_1)^{e_1'} \cdots (x- \xi_n)^{e_n'},
	\]
	where each $\xi_k \in V/\m$ are nonzero and $i, j, e_k, e_k' \in \N$ for $k = 1,2, \dots, n$. 
	
	Now by Proposition \ref{Prop:FormOfMinval}, we note that there exists some $\epsilon' \in \Gamma$ with $0 < \epsilon' \leq \epsilon$ such that
	\[
	\minval_f(\gamma) = \begin{cases}
		d_1\gamma + \delta_1, & \text{ if $\alpha - \epsilon' \leq \gamma \leq \alpha$,}\\
		d_2\gamma + \delta_2, & \text{ if $\alpha \leq \gamma \leq \alpha + \epsilon'$}
	\end{cases}
	\]
	and
	\[
	\minval_g(\gamma) = \begin{cases}
		d_1'\gamma + \delta_1', & \text{ if $\alpha - \epsilon' \leq \gamma \leq \alpha$,}\\
		d_2'\gamma + \delta_2', & \text{ if $\alpha \leq \gamma \leq \alpha + \epsilon'$,}
	\end{cases}
	\]
	for some $d_1, d_2, d_1', d_2' \in \N$ and $\delta_1, \delta_2, \delta_1', \delta_2' \in \Gamma$. From this, we get that
	\[
	d_1 - d_1' = c_1 \quad \text{and} \quad d_2 - d_2' = c_2
	\]
	by Lemma \ref{Lem:PowersOfLoc}. We also have
	\[
	(i + e_1 + \cdots + e_n) - i = d_1 - d_2 \quad \text{and} \quad (j + e_1' + \cdots + e_n') - j = d_1' - d_2'
	\]
	by Proposition \ref{Prop:FormOfMinval}. Putting it all together, we get that
	\[
	(e_1 + \cdots + e_n ) - (e_1' + \cdots + e_n') = c_1 - c_2 > 0. 
	\]
	This means that there exists $\ell \in \{1, \dots, n\}$ such that $e_\ell > e_\ell'$. 
	
	Let $u \in V$ be such that $u + \m = \xi_\ell$. We will lift $\loc_{f,b}(x)$ and $\loc_{g,b}(x)$ back to $V[x]$. We now have
	\[
	\frac{f(bx)}{t} = (x-u)^{e_\ell}f_1(x) + f_2(x) \quad \text{and} \quad \frac{g(bx)}{t'} = (x-u)^{e_\ell'}g_1(x) + g_2(x),
	\]
	where $t, t' \in K$ such that $v(t) = \minval_f(v(b))$ and $v(t') = \minval_g(v(b))$,  $f_1, g_1 \in V[x]$ such that $f_1(u), g_1(u) \notin \m$, and $f_2, g_2 \in \m[x]$. 
	
	Let $h \in \m$ such that $v(h) < \frac{1}{e_\ell}\min\{\minval_{f_2}(0), \minval_{g_2}(0) \}$. This is possible since $\m$ is not principal. Set $a \coloneqq b(u+h)$. Note that $v(a) = v(b) = \alpha$. We then calculate
	\[
	\varphi(a) = \frac{t(h^{e_\ell}f_1(u+h) + f_2(u+h))}{t'(h^{e_\ell'}g_1(u+h) + g_2(u+h))}.
	\]
	Then, $v(\varphi(a)) = \minval_{\varphi}(\alpha) + (e_\ell - e_\ell')v(h) > \minval_{\varphi}(\alpha)$, as desired. 
	
	If $c_1 < c_2$, then apply the $c_1 > c_2$ case to $\frac{1}{\varphi}$ to get the desired result.
\end{proof}

We use the previous lemma to determine that $\IntR(V)$ is not a Prüfer domain for $V$ a valuation domain with algebraically closed residue field and maximal ideal that is not principal.

\begin{theorem}\label{Thm:IntROverNonMonicNonSingularValDomIsNotPrufer}
	Suppose that $V/\m$ is algebraically closed and $\m$ is not a principal ideal of $V$. Then $\IntR(V)$ is not Prüfer.
\end{theorem}

\begin{proof}
	Aiming for a contradiction, we assume that $\IntR(V)$ is Prüfer.
	
	Let $d \in \m$. Since $\IntR(V)$ is Prüfer, the finitely-generated ideal $(x, d)$ is invertible. This means that there are $\varphi, \psi \in (x, d)^{-1}$ such that $x\varphi + d \psi =1$. By Proposition \ref{Prop:AttainMinvalRat}, there exists some $b \in V$ such that $v(b) = v(d)$ and $v(\varphi(b)) = \minval_\varphi(v(d))$ and $v(\psi(b)) = \minval_{\psi}(v(d))$. Evaluating $x\varphi + d \psi =1$ at $x = b$, we obtain
	\[
	b\varphi(b) + d\psi(b) = 1.	
	\]
	
	We have $x\varphi, d\psi \in \IntR(V)$, so $b\varphi(b), d\psi(b) \in V$. Thus, we have $v(b\varphi(b)) = 0$ or $v(d\psi(b)) = 0$. Then we get $\minval_\varphi(v(d)) = v(\varphi(b)) = -v(b) = -v(d)$ or $\minval_{\psi}(v(d)) = v(\psi(b)) = -v(d)$. 
	
	Either way, we have some function $\rho \in (x, d)^{-1}$ such that  $v(b\rho(b)) = 0$ and $v(\rho(b)) = \minval_{\rho}(v(d)) = -v(d)$. By Proposition \ref{Prop:FormOfMinvalRat}, there exists some $\varepsilon \in \Q\Gamma$ with $\varepsilon > 0$ such that there exist some $c, c' \in \Z$ and $\beta, \beta' \in \Gamma$ so that
	\[
	\minval_\rho(\gamma) = \begin{cases}
		c\gamma + \beta, & v(d) - \epsilon \leq \gamma \leq v(d), \\
		c'\gamma + \beta', & v(d) \leq \gamma \leq v(d) + \epsilon.
	\end{cases}
	\]
	Since $\m$ is not principal, due to Lemma \ref{Lem:MinvalEqualityAE}, there exists $b' \in V$ with the property that $v(d) - \epsilon < v(b') < v(d)$ and $v(\rho(b')) = cv(b') + \beta$. Since $x\rho \in \IntR(V)$, we have 
	\[
	v(b'\rho(b')) = v(b') + cv(b') + \beta = (c+1)v(b') + \beta \geq 0 = v(b\rho(b)) = (c+1)v(b) + \beta.
	\]
	This implies $(c+1)v(b') \geq (c+1)v(b)$ and thus $0 \geq (c+1)(v(b) - v(b'))$. We know that $v(b) > v(b')$ so we must have $c + 1 \leq 0$. In other words, $c \leq -1$. 
	
	Now there exists $b'' \in V$ so that $v(d) < v(b'') < v(d) + \epsilon$ and $v(\rho(b'')) = c'v(b'') + \beta'$. Because $d\rho \in \IntR(V)$, we get that 
	\[
	v(d\rho(b''))) = v(d) + c'v(b'') + \beta' \geq 0 = v(b\rho(b)) =  v(d) + c'v(b) + \beta'.	
	\]
	Thus, $c'v(b'') \geq c'v(b)$. This implies that $c' \geq 0$ since $v(b'') > v(b)$.

	Since $c < c'$, there exists $a \in V$ such that $v(\rho(a)) < \minval_\rho(v(d)) = -v(d)$ and $v(a) = v(d)$ according to Lemma \ref{Lem:ValuationLemma}. This implies that $v(a\rho(a)) < 0$, contradicting the fact that $x\rho \in \IntR(V)$. We can conclude that $(x, d)$ is not invertible, so $\IntR(V)$ cannot be Prüfer.
\end{proof}

This result combined with the results of \cite{PruferNonDRings, IntValuedRational} completely classifies the case of when $\IntR(V)$ is a Prüfer domain given that $V$ is a valuation domain. 

\begin{corollary}\label{Cor:IntRVClassification}
	Let $V$ be a valuation domain. Then $\IntR(V)$ is a Prüfer domain if and only if $V/\m$ is not algebraically closed or $\m$ is a principal ideal of $V$.
\end{corollary}

In fact, for a valuation domain $V$, most of the time when $\IntR(V)$ is Prüfer, the ring $\IntR(V)$ is also Bézout. We know that $\IntR(V)$ is Bézout when $V$ has a principal maximal ideal or there exist two nonconstant, monic, unit-valued polynomials over $V$ of coprime degrees \cite[Theorem 3.5 and Corollary 3.3]{IntValuedRational}. Note that the latter condition is equivalent to saying that there exist two nonconstant polynomials of coprime degrees over $V/\m$ with no roots in $V/\m$, where $\m$ is the maximal ideal of $V$. We will now completely characterize when the ring integer-valued rational functions over a valuation domain is a Bézout domain. We first require a lemma about the minimum valuation functions of generators of finitely-generated ideals in $\IntR(V)$. 

\begin{lemma}\label{Lem:MinMinval}
	Suppose $V/\m$ is infinite. Let $\varphi_1, \dots, \varphi_n, \psi_1, \dots, \psi_m \in \IntR(V)$ be nonzero integer-valued rational functions such that
	\[
	(\varphi_1, \dots, \varphi_n) = (\psi_1, \dots, \psi_m)
	\]
	as ideals of $\IntR(V)$. Then
	\[
	\min\{ \minval_{\varphi_1}(\gamma), \dots, \minval_{\varphi_n}(\gamma) \} = \min\{ \minval_{\psi_1}(\gamma), \dots, \minval_{\psi_m}(\gamma) \}
	\]
	for all $\gamma \in \Gamma$ such that $\gamma \geq 0$. 
\end{lemma}

\begin{proof}
	Let $\gamma \in \Gamma$ with $\gamma \geq 0$. Since $V$ has an infinite residue field, there exists $d \in V$ such that $v(d) = \gamma$ and $\minval_{\varphi_i}(\gamma) = v(\varphi_i(d)), \minval_{\psi_j}(\gamma) = v(\psi(d))$
	for all $i$ and $j$ by Proposition \ref{Prop:AttainMinvalRat}. Because $(\varphi_1(d), \dots, \varphi_n(d)) = (\psi_1(d), \dots, \psi_m(d))$, it follows that
	\[
	\min\{ v(\varphi_1(d)), \dots, v(\varphi_n(d)) \} = \min\{ v(\psi_1(d)), \dots, v(\psi_m(d)) \}.
	\]

	Therefore, we obtain
	\[
	\min\{ \minval_{\varphi_1}(\gamma), \dots, \minval_{\varphi_n}(\gamma) \} = \min\{ \minval_{\psi_1}(\gamma), \dots, \minval_{\psi_m}(\gamma) \}.
	\]
\end{proof}

Now we characterize when $\IntR(V)$ is a Bézout domain.

\begin{proposition}\label{Prop:MonicNotBezoutCase}
	Suppose that $\m$ is not principal and there does not exist two nonconstant polynomials of coprime degrees over $V/\m$ with no roots in $V/\m$. Then $\IntR(V)$ is not Bézout. 
\end{proposition}

\begin{proof}
	Let $t \in \m$. We want to show that the finitely-generated ideal $(x,t)$ of $\IntR(V)$ is not principal. Suppose on the contrary that $(x,t) = (\varphi)$ for some $\varphi \in \IntR(V)$.

	Note that $V/\m$ is necessarily an infinite field, so by Lemma \ref{Lem:MinMinval}, 
	\[
	\minval_\varphi(\gamma) = \min\{\minval_x(\gamma), \minval_t(\gamma) \} = \begin{cases}
		\gamma, & \text{if $0 \leq \gamma \leq v(t)$},\\
		v(t), & \text{if $\gamma \geq v(t)$},
	\end{cases}
	\]
	for each $\gamma \in \Gamma$ such that $\gamma \geq 0$. Therefore, if we write $\varphi = \frac{f}{g}$ for some $f, g \in V[x]$, then $\deg (\loc_{f, t}) = \deg (\loc_{g, t}) + 1$ and the degree of the lowest degree monomial of $\loc_{f, t}$ and $\loc_{g, t}$ are the same by Lemma \ref{Lem:PowersOfLoc}. Plus, we claim it is impossible for every nonzero element of $V/\m$ to be a root of $\loc_{f, t}$ and $\loc_{g, t}$ of the same multiplicity. Here, we allow for the possibility that an element is a root of multiplicity $0$, meaning it is not a root. Suppose $\xi_1, \dots, \xi_n$ are the nonzero roots of $\loc_{f, t}$ and $\loc_{g, t}$ with each $\xi_i$ appearing in $\loc_{f, t}$ and $\loc_{g, t}$ with multiplicity $e_i$. Then the polynomials $F(x) = \frac{\loc_{f, t}(x)}{x^m(x-\xi_1)^{e_1}\cdots(x-\xi_n)^{e_n}}$ and $G(x) = \frac{\loc_{g, t}(x)}{x^m(x-\xi_1)^{e_1}\cdots(x-\xi_n)^{e_n}}$ both have no roots over $V/\m$. Moreover, since $\deg(F) = \deg(G) + 1 \geq 3$, we know that $\gcd(\deg(F), \deg(G)) = 1$. Thus, the assumption about $V/\m$ is contradicted.  
	
	Thus, there exists some nonzero element $\xi \in V/\m$ such that $\xi$ is a root of $\loc_{f, t}$ of multiplicity $c_1$ and $\xi$ is a root of $\loc_{g, t}$ of multiplicity $c_2$ with $c_1 \neq c_2$.  Let $u \in V$ be a lift of $\xi$. Then by lifting $\loc_{f, t}$, we obtain 
	\[
	\frac{f(tx)}{b} = (x-u)^{c_1}f_1(x) + f_2(x),
	\]
	where $b \in V$ is some element such that $v(b) = \minval_f(v(t))$, $f_1(x) \in V[x]$ is such that $f_1(u) \notin \m$, and $f_2(x) \in \m[x]$. We similarly obtain $\frac{g(tx)}{b'} = (x-u)^{c_2}g_1(x) + g_2(x)$, where $b' \in V$ is such that $v(b') = \minval_g(v(t))$, $g_1(x) \in V[x]$ is such that $g_1(u) \notin \m$, and $g_2(x) \in \m[x]$. Then there exists an element $h \in \m$ such that $v(h) < \frac{1}{\max\{c_1, c_2\}}\min\{\minval_{f_2}(0), \minval_{g_2}(0) \}$ and 
	\[
	v\left(\frac{f(t(u+h))}{b} \right) = v( h^{c_1} f_1(u+h) + f_2(u+h)) = v( h^{c_1} f_1(u+h)) = c_1v(h). 
	\]
	Therefore, $v(f(t(u+h))) = c_1v(h) + \minval_f(v(t))$. A similar calculation yields $v(g(t(u+h))) = c_2v(h) + \minval_g(v(t))$. Now, we have \[v(\varphi(t(u+h))) = (c_1-c_2)v(h) + \minval_\varphi(v(t)) = (c_1-c_2)v(h)+v(t).\] 
	
	We must have $v(\varphi(t(u+h))) = \min\{ v(t(u+h)), v(t) \} = v(t)$ because $\varphi$ generates $(x, t)$. This implies that $c_1 = c_2$, a contradiction. Thus, the existence of $\varphi$ is impossible, meaning that $\IntR(V)$ cannot be Bézout.
\end{proof}

This, along with \cite[Theorem 3.5 and Corollary 3.3]{IntValuedRational}, gives us a complete characterization of when $\IntR(V)$ is a Bézout domain.

\begin{corollary}
	The ring $\IntR(V)$ is a Bézout domain if and only if $\m$ is principal or there exist two nonconstant polynomials of coprime degrees over $V/\m$ with no roots in $V/\m$.
\end{corollary}

\subsection{When $\IntR(V)$ is not a Prüfer domain}

We have completely classified the conditions on $V$ that make $\IntR(V)$ Prüfer. We consider the case when $\IntR(V)$ is not Prüfer and try to understand in what ways $\IntR(V)$ fails to be Prüfer. 

If $\IntR(V)$ is not Prüfer, then we know that $V$ has algebraically closed residue field and maximal ideal that is not principal. We first consider the case when we additionally assume that the value group that is not divisible. Since $\IntR(V)$ is not Prüfer, there must be some prime ideal $\mathfrak{P}$ of $\IntR(V)$ such that $\IntR(V)_\mathfrak{P}$ is not a valuation domain. We will show that such a prime ideal $\mathfrak{P}$ cannot be a pointed maximal ideal. First, we will need a lemma.

\begin{lemma}\label{Lem:DivideBoth}
	Let $D$ be a domain and $\p$ be a prime ideal. Let $a, b$ be two elements of $K$, the field of fractions of $D$, such that $b \neq 0$. Then $\frac{a}{b} \in D_\p$ if and only if there exists $c \in K$ such that $\frac{a}{c} \in D$ and $\frac{b}{c} \in D \setminus \p$. 
\end{lemma}

\begin{proof}
	Suppose $\frac{a}{b} \in D_\p$. This implies that $\frac{a}{b} = \frac{r}{s}$ for some $r \in D$ and $s \in D \setminus \p$. Now set $c\coloneqq\frac{b}{s}$. We see that $\frac{a}{b/s} = r \in D$ and $\frac{b}{b/s} = s \in D\setminus \p$. 
	
	On the other hand, if there exists $c \in K$ such that $\frac{a}{c} \in D$ and $\frac{b}{c} \in D \setminus \p$, then $\frac{a}{b} = \frac{a/c}{b/c} \in D_\p$. 
\end{proof}

We will show that for a valuation domain $V$ with algebraically closed residue field, maximal ideal that is not principal, and value group that is not divisible that $\IntR(V)$ localized at a pointed maximal ideal is a valuation domain. This is true in a more general setting so we give the result with weaker assumptions. 

\begin{proposition}\label{Prop:PointedMaximalIdealsAreEssential}
	Suppose that $\Gamma$ is not divisible. Let $E$ be a subset of $K$ and take $a \in E$. Then 
	\[\IntR(E, V)_{\mathfrak{M}_{\m, a}} = \{\varphi \in K(x) \mid \varphi(a) \in V \},\] a valuation domain. 
\end{proposition}

\begin{proof}

	It suffices to assume without loss of generality that $0 \in E$ and show that $\IntR(E, V)_{\mathfrak{M}_{\m, 0}} = \{\varphi \in K(x) \mid \varphi(0) \in V \}$ since $\IntR(E, V) \cong \IntR(E-a, V)$ for all $a \in E$. 
	
	We see that $\IntR(E, V)_{\M_{\m, 0}} \subseteq \{\varphi \in K(x) \mid \varphi(0) \in V \}$, so we want to show the reverse inclusion. Let $\varphi \in K(x)$ be a nonzero rational function such that $\varphi(0) \in V$. We can write $\varphi = \frac{f}{g}$ such that $f, g \in K[x]$ and $v(g(0)) = 0$. 
	
	Since $v(g(0)) = 0$, we know that $v(g(a)) = 0$ for all $a \in K$ such that $v(a)$ is sufficiently large due to the fact that the valuation of each monomial in $g(a)$ except the constant can be arbitrarily large depending on $v(a)$. We also know that by Corollary \ref{Cor:MinvalAlmostEverywhere} that $v(f(a)) = \minval_f(v(a))$ for all $a \in K$ such that $v(a)$ is sufficiently large. Since $v(f(0)) \geq 0$, we can ensure that $v(f(a)) \geq 0$ for all $a \in K$ such that $v(a)$ is sufficiently large. The reasoning for this is similar to that for $g(x)$, except $f(x)$ might not have a constant term. From this, we deduce that there exists $\delta \in \Gamma$ with $\delta \geq 0$ such that $v(f(d)) \geq v(g(d)) = 0$ for all $d \in K$ with $v(d) > \delta$.
	
	Since $\Gamma$ is not divisible, there exists $\eta \in \Q\Gamma \setminus \Gamma$ such that $\eta > \delta$. There also exists an $n \in \N$ with $n > 0$ such that $n\eta \in \Gamma$. Let $m\coloneqq \max\{\deg f, \deg g \}$. Set $h(x) \coloneqq \frac{1}{t}x^{mn} + 1$, where $t \in V$ is such that $v(t) = mn\eta$. We want to show that $\frac{f}{h}, \frac{g}{h} \in \IntR(E,V)$ and $\frac{g}{h} \notin \M_{\m, 0}$.
	
	Now let $d \in E$. Then
	\[
	v(h(d)) = \begin{cases}
		0, & \text{if $v(d) > \eta$},\\
		mn(v(d)-\eta), & \text{if $v(d) < \eta$}. 
	\end{cases}.
	\]
	From this, we can gather that for all $d \in E$ with $v(d) > \eta$, we have $v\left(\frac{f}{h}(d)\right) \geq 0$ and $v\left(\frac{g}{h}(d)\right) \geq 0$. Furthermore, $\frac{g}{h}(0) = g(0) \notin \m$. 
	
	Now fix $d \in E$ such that $v(d) < \eta$. We know  $v(f(d)) \geq \minval_f(v(d)) = cv(d) + \beta$ for some $c \in \Z$ and $\beta \in \Gamma$. Notice that $c\eta + \beta \geq \minval_f(\eta)$ by the definition of $\minval_f$ and thus $c\eta +\beta \geq \minval_f(\eta) \geq 0$. This implies $\beta \geq -c\eta$. Now we obtain
	\[
	v(f(d)) \geq cv(d) + \beta \geq c(v(d) - \eta) > mn(v(d) - \eta) = v(h(d))
	\]
	because $c \leq m < mn$ and $v(d) - \eta < 0$. This shows that $\frac{f}{h}(d) \in \m$. We can show that $\frac{g}{h}(d) \in \m$ the same way. Thus, $\frac{f}{h}, \frac{g}{h} \in \IntR(E,V)$ and $\frac{g}{h} \notin \M_{\m, 0}$, which implies that $\frac{f}{g} = \varphi \in \IntR(E, V)_{\mathfrak{M}_{\m, a}}$. 
	
	This shows that $\IntR(E, V)_{\mathfrak{M}_{\m, a}} = \{\varphi \in K(x) \mid \varphi(a) \in V \}$, so $ \IntR(E, V)_{\mathfrak{M}_{\m, a}}$ is a valuation domain.

\end{proof}

Suppose we have a valuation domain $V$ such that $\m$ is not a principal ideal of $V$ and $V/\m$ is algebraically closed. We know that $\IntR(V)$ is not Prüfer, so there must be some prime ideal $\mathfrak{P}$ of $\IntR(V)$ such that $\IntR(V)_\mathfrak{P}$ is not a valuation domain. We have just seen that this ideal cannot be a pointed maximal ideal if $\Gamma$ is not divisible. We will give an example of such a prime in the form of an ultrafilter limit of pointed maximal ideals. We will also use the idea of a pseudo-divergent sequence \cite[Definition 2.1]{Peruginelli}. Note that we do not need to assume that the value group of $V$ is not divisible. 

\begin{proposition}
	Suppose that $\m$ is not a principal ideal of $V$ and $V/\m$ is algebraically closed. Fix some $d \in V$ with $v(d) > 0$. Let $\{d_i\}_{i=1}^\infty \subseteq V$ such that $v(d_i) > v(d)$ for each $i$ and for each $\varepsilon \in \Gamma$ with $\varepsilon > 0$, there exists an $i$ such that $v(d_i) < v(d) + \varepsilon$. Then let $\mathcal{U}$ be a non-principal ultrafilter of $\{\M_{\m, d_i}\}_{i=1}^\infty$. Then $\lim\limits_{\mathcal{U}} \M_{\m, d_i}$ is a prime ideal of $\IntR(V)$ such that $\IntR(V)_{\lim\limits_{\mathcal{U}} \M_{\m, d_i}}$ is not valuation domain.
\end{proposition}

\begin{proof}
	First note that the existence of $\{d_i\}_{i=1}^\infty \subseteq V$ such that $v(d_i) > v(d)$ for each $i$ and for each $\varepsilon \in \Gamma$ with $\varepsilon > 0$, there exists an $i$ such that $v(d_i) < v(d) + \varepsilon$ depends on the maximal ideal not being principal.

	We will show that $\frac{x+d}{d} \notin \IntR(V)_{\lim\limits_{\mathcal{U}} \M_{\m, d_i}}$. A similar argument will show that $\frac{d}{x+d} \notin \IntR(V)_{\lim\limits_{\mathcal{U}} \M_{\m, d_i}}$. From this, we will see that $\IntR(V)_{\lim\limits_{\mathcal{U}} \M_{\m, d_i}}$ is not valuation domain.
	
	Suppose on the contrary that $\frac{x+d}{d} \in \IntR(V)_{\lim\limits_{\mathcal{U}} \M_{\m, d_i}}$. Then we may write $\frac{x+d}{d} = \frac{\varphi}{\psi}$, where $\varphi, \psi \in \IntR(V)$ and $\psi \notin \lim\limits_{\mathcal{U}} \M_{\m, d_i}$. Set $\rho = \frac{d}{\psi}$. Then $\frac{x+d}{\rho} = \varphi$ and $\frac{d}{\rho} = \psi$.

	We then get that $\{\M_{\m, d_i} \mid i\in \N, \psi \notin \M_{\m, d_i} \} \in \mathcal{U}$. Because $\mathcal{U}$ is not principal, $\{\M_{\m, d_i} \mid i\in \N, \psi \notin \M_{\m, d_i} \}$ is an infinite set. Therefore, there exists $i \in \N$ such that $v(d_i)$ is arbitrarily close to $v(d)$ and $\psi \notin \M_{\m, d_i}$, or equivalently, $v(\psi(d_i)) = 0$. This shows that there exists some $\epsilon \in \Gamma$ with $\epsilon > 0$ such that $\minval_\psi(\gamma) = 0$ for $\gamma \in \Gamma$ such that $v(d) \leq \gamma \leq v(d) + \epsilon$. Since $\minval_\rho = \minval_d - \minval_\psi$, we obtain that $\minval_\rho(\gamma) = v(d)$ for $\gamma$ such that $v(d) \leq \gamma \leq v(d) + \epsilon$. 
	
	On the other hand, we can make $\epsilon$ small enough so that $\minval_\rho(\gamma) = c\gamma + \beta$ for $\gamma$ such that $v(d) - \epsilon \leq \gamma \leq v(d)$ for some $c \in \Z$ and $\beta \in \Gamma$. We know that $\minval_\rho(v(d)) = v(d)$, so $\beta = (1-c)v(d)$. Next, since $\varphi \in \IntR(V)$, for $\gamma \in \Gamma$ with $\gamma \geq 0$, we have that $\minval_{x+d}(\gamma) - \minval_\rho(\gamma) = \minval_\varphi(\gamma) \geq 0$. Lemma \ref{Lem:MinvalEqualityAE} implies this inequality for almost all such values of $\gamma$, and by the form of the minimum valuation function given in Proposition \ref{Prop:FormOfMinvalRat}, the inequality holds for all $\gamma \geq 0$. Thus, for $\gamma$ such that $v(d) - \epsilon < \gamma \leq v(d)$, we get \[\gamma = \minval_{x+d}(\gamma) \geq \minval_\rho(\gamma) = c\gamma + (1-c)v(d).\] This implies that $(1-c)\gamma \geq (1-c)v(d)$, but we have $\gamma \leq v(d)$, so it must be the case that $1-c \leq 0$, or equivalently, $c \geq 1$.

	By Lemma \ref{Lem:ValuationLemma}, there exists $a \in V$ with $v(a) = v(d)$ such that $v(\rho(a)) > v(d)$, but then $v(\psi(a)) = v\left(\frac{d}{\rho(a)}\right) < 0$, contradicting the fact that $\psi \in \IntR(V)$. Thus, $\frac{x+d}{d} \notin \IntR(V)_{\lim\limits_{\mathcal{U}} \M_{\m, d_i}}$.

\end{proof}

\begin{remark}
	Suppose that $\m$ is not a principal ideal of $V$, that $V/\m$ is algebraically closed, and that $\Gamma$ is not divisible. Then $\IntR(V)$ is an example of an essential domain, a domain that can be written as the intersection of some family of essential valuation overrings, that is not a P$v$MD. Another example of an essential domain that is not a P$v$MD can be found in \cite{HeinzerOhm}. 
	
	We can write
	\[
	\IntR(V) = \bigcap_{a \in V} \IntR(V)_{\M_{\m, a}}.
	\]
 For every $a \in V$, we know that $\M_{\m, a}$ is essential by Proposition \ref{Prop:PointedMaximalIdealsAreEssential}. This means that $\IntR(V)$ is an essential domain. Furthermore, for every $a \in V$, the ideal $\M_{\m, a}$ being essential implies that $\M_{\m, a}$ is a $t$-ideal \cite[Lemma 3.17]{Kang}. Using the notation of the previous proposition, we know that $\lim\limits_{\mathcal{U}} \M_{\m, d_i}$ is a t-ideal since the ultrafilter limit of t-ideals is a t-ideal \cite[Proposition 2.5]{PvMD}. However, the ideal $\lim\limits_{\mathcal{U}} \M_{\m, d_i}$ is a $t$-maximal ideal of $\IntR(V)$ that is not essential. Thus, $\IntR(V)$ is not a P$v$MD.
\end{remark}

We will now consider the case where $V/\m$ is algebraically closed and the value group is divisible (which implies that $\m$ is not principal). We can actually detect that $\IntR(V)$ is not Prüfer by localizing at a pointed maximal ideal.

\begin{proposition}\label{Prop:AlgClosedResFieldDivisibleValueGroup}
	Suppose that $V/\m$ is algebraically closed and $\Gamma$ is divisible. Then the localization of $\IntR(V)$ at any pointed maximal ideal is not a valuation ring. 
\end{proposition}

\begin{proof}
	Let $a \in V$. Mapping $x \mapsto x-a$ and fixing $V$ determines an automorphism for $\IntR(V)$ for any $a \in V$, we can study the behavior of localizing $\IntR(V)$ at $\M_{\m, a}$ by only considering the localization at $\M_{\m, 0}$. 
	
	Suppose for a contradiction that $W \coloneqq \IntR(V)_{\M_{\m, 0}}$ is a valuation domain. Fix a nonzero $d \in \m$. We have $\frac{d}{x+d} \in W$ or $\frac{x+d}{d} \in W$. Thus, by Lemma \ref{Lem:DivideBoth}, we have some $\varphi \in K(x)$ such that $\frac{d}{\varphi} \in \IntR(V)$ and $\frac{x+d}{\varphi} \in \IntR(V)$, and additionally, $\frac{d}{\varphi} \notin \M_{\m, 0}$ or $\frac{x+d}{\varphi} \notin \M_{\m, 0}$. 
	
	We know that $\minval_\varphi$ has the form
	\[
	\minval_\varphi(\gamma) =
	\begin{cases}
		c_1\gamma + \beta_1, & \gamma \leq \delta_1,\\
		c_2\gamma + \beta_2, & \delta_1 \leq \gamma \leq \delta_2,\\
		\vdots\\
		c_{n-1}\gamma + \beta_{n-1}, & \delta_{n-2} \leq \gamma \leq \delta_{n-1},\\
		c_n\gamma + \beta_n, & \delta_{n-1} \leq \gamma, 
	\end{cases}
	\]
	for some $c_i \in \Z$, $\beta_i \in \Gamma$, $\delta_i \in \Q\Gamma$ by Proposition \ref{Prop:FormOfMinvalRat}. We may choose $\delta_{n-1}$ such that $c_{n-1} \neq c_n$. We claim that we must have $c_n = 0$ and $\beta_n = v(d)$. If $\frac{x+d}{\varphi} \notin \M_{\m, 0}$, then $v\left(\frac{0+d}{\varphi(0)} \right) = 0$, so $v(\varphi(0)) = v(d)$. If $\frac{d}{\varphi} \notin \M_{\m, 0}$, we similarly have $v(\varphi(0)) = v(d)$. Either way, we have $v(\varphi(0)) = v(d)$. Since $d \neq 0$, we have that $\varphi(0) \neq 0$. Then $v(\varphi(a)) = v(\varphi(0))$ for $a \in K$ such that $v(a)$ is sufficiently large. This implies that $\minval_\varphi(\gamma)$ is constant and equal to $v(\varphi(0)) = v(d)$ for $\gamma \in \Gamma$ sufficiently large by Lemma \ref{Lem:MinvalEqualityAE}. Therefore, $c_n = 0$ and $\beta_n = v(d)$.

	We also claim that $\delta_{n-1} \geq v(d)$. If not, by Lemma \ref{Lem:MinvalEqualityAE}, there exists $b \in V$ such that $\max\{\delta_{n-1},0 \} < v(b) < v(d)$ and $v(\varphi(b)) = \minval_{\varphi}(v(b)) = v(d)$. However, this would imply that
	\[
	v\left(\frac{b+d}{\varphi(b)}\right) = v(b) - v(d) < 0,
	\]
	contradicting the fact that $\frac{x+d}{\varphi} \in \IntR(V)$. 
	
	Next, we claim that $c_{n-1} > c_n = 0$. If $c_{n-1} < 0$, we can find $b \in V$ such that $\max\{\delta_{n-2},0 \} < v(b) < \delta_{n-1}$ and $v(\varphi(b)) = \minval_\varphi(v(b)) = c_{n-1}v(b) + \beta_{n-1}$. Since $c_{n-1}\delta_{n-1} + \beta_{n-1} = v(d)$, we have that $v(\varphi(b)) - v(d) = c_{n-1}(v(b) - \delta_{n-1}) > 0$ as both $c_{n-1}$ and $v(b)-\delta_{n-1}$ are less than 0. This contradicts the fact that $\frac{d}{\varphi} \in \IntR(V)$.
	
	Because $\Gamma$ is divisible, we know that $\delta_{n-1} \in \Gamma$. Furthermore, since $c_{n-1} > c_n$, by Lemma \ref{Lem:ValuationLemma}, there exists some element $a \in V$ so that we have $v(a) = \delta_{n-1}$ and $v(\varphi(a)) > \minval_\varphi(\delta_{n-1}) = v(d)$. This contradicts $\frac{d}{\varphi} \in \IntR(V)$. Thus, the assumption that $W$ is a valuation domain is false. 
\end{proof}

	We end this section by noting that for a valuation domain $V$ such that $V/\m$ is algebraically closed and $\m$ is not a principal ideal of $V$, even though $\IntR(V)$ is not a Prüfer domain, there are subsets $E$ of the field of fractions $K$ of $V$ such that $\IntR(E,V)$ is a Prüfer domain. For example, if $E$ is a singleton, then $\IntR(E,V)$ is a valuation domain and therefore a Prüfer domain. Likewise, there are other subsets $E$ such that $\IntR(E,V)$ is not a Prüfer domain. We can obtain from the proofs of Theorem \ref{Thm:IntROverNonMonicNonSingularValDomIsNotPrufer} and Lemma \ref{Lem:ValuationLemma} some conditions $E$ such that $\IntR(E,V)$ is not Prüfer. One case is indicated in the following result. 
	
	\begin{proposition}\label{Prop:IntRDVNotPruferCase}
		Suppose $V/\m$ is algebraically closed and $\m$ is not a principal ideal of $V$. Also suppose that $V$ is an essential valuation overring of a domain $D$ centered on a maximal ideal of $D$. Then $\IntR(D, V)$ is not Prüfer. 
	\end{proposition}

\begin{proof}
	Let $v$ be the valuation associated with $V$ and $\Gamma$ the value group. This follows from the proofs of Theorem \ref{Thm:IntROverNonMonicNonSingularValDomIsNotPrufer} and Lemma \ref{Lem:ValuationLemma} and the facts that $V/\m \cong D/(D\cap \m)$ and for every $\gamma \in \Gamma$, there exists $d \in D$ such that $v(d) = \gamma$.  
\end{proof}

\section{Integer-valued rational functions over Prüfer domains}

\indent\indent Let $D$ be a domain. We want to know what conditions on $D$ makes $\IntR(D)$ a Prüfer domain. The case for rings of integer-valued polynomials has been answered \cite{IntPruferClassification, CahenChabertFrisch}. If $\IntR(D)$ is a Prüfer domain, then $D$ is a Prüfer domain since homomorphic images of Prüfer domains are Prüfer domains. In \cite{IntValuedRational}, we see that if $D$ is a monic or singular Prüfer domain, then $\IntR(D)$ is Prüfer. We want to first investigate a few cases when $D$ is a Prüfer domain but $\IntR(D)$ is not a Prüfer domain. 

One possible obstruction to $\IntR(E, D)$ being a Prüfer domain is an valuation overring $V$ of $D$ that yields a domain $\IntR(E, V)$ that is not Prüfer. 

\begin{proposition}
	Let $D$ be a domain with field of fractions $K$ and $E \subseteq K$ a subset. Also let $V$ be a valuation overring of $D$ such that $\IntR(E,D)$ and $\IntR(E,V)$ have the same field of fractions. If $\IntR(E,V)$ is not Prüfer, then $\IntR(E,D)$ is not Prüfer either. 
\end{proposition}

\begin{proof}
	Since $D\subseteq V$, we have $\IntR(E,D) \subseteq \IntR(E,V)$, so $\IntR(E,V)$ is an overring of $\IntR(E,D)$ that is not Prüfer. Thus, $\IntR(E,D)$ is not Prüfer as every overring of a Prüfer domain is a Prüfer domain. 
\end{proof}

\begin{corollary}\label{Cor:ValOverringObstruction}
	Let $D$ be a domain with field of fractions $K$. Suppose there exists a valuation overring $V$ of $D$ centered on a maximal ideal of $D$ such that the residue field of $V$ is algebraically closed and the maximal ideal of $V$ is not principal. Then $\IntR(D)$ is not Prüfer. 
\end{corollary}

\begin{proof}
	Since $D \subseteq \IntR(D)$ and $x \in \IntR(D)$, we have that the field of fractions of $\IntR(D)$ is $K(x)$. Thus, $\IntR(D,V)$ is an overring of $\IntR(D)$. Moreover, $\IntR(D, V)$ is not Prüfer by Proposition \ref{Prop:IntRDVNotPruferCase}, so $\IntR(D)$ is not Prüfer either. 
\end{proof}

Given a Prüfer domain $D$, obstructions to $\IntR(D)$ being a Prüfer domain do not necessarily come locally from a single valuation overring. The obstruction can come from a collection of valuation overrings that behave collectively like a valuation overring with algebraically closed residue field and maximal ideal that is not principal, but individually, each valuation overring does not have both algebraically closed residue field and maximal ideal that is not principal. The construction of such domains is done via sequential domains, a generalization of sequence domains in \cite{SeqDom}.

\begin{definition}
	Let $D$ be a domain with field of fractions $K$. We say that $D$ is a \textbf{sequential domain} if there exist for each $i \in \N \setminus \{0\}$, valuations $v_i: K \to \Gamma_i \cup \{\infty\}$, where $\Gamma_i$ is a totally ordered abelian group, such that
	\begin{itemize}
		\item each associated valuation domain $V_i$ is an essential overring of $D$ such that $D = \bigcap_{i=1}^\infty V_i$,
		\item there is a common totally ordered abelian group $\Gamma$ with embeddings $\Gamma_i \hookrightarrow \Gamma$ such that for each $d \in D$, $\{v_i(d)\}_{i=1}^\infty$ is eventually constant viewed as a sequence in $\Gamma$, and
		\item there exists $\varpi \in D$ such that $v_i(\varpi)$ is not eventually 0. 
	\end{itemize}
	
	Since $v_i(d)$ is eventually constant for all $d \in D$, there is valuation $v_\infty$ defined by $v_\infty(d) = \lim\limits_{i\to\infty} v_i(d)$ with associated valuation domain $V_\infty$ and maximal ideal $\m_\infty$. We also have an embedding of the value group $\Gamma_\infty \hookrightarrow \Gamma$. 
	
	We say a sequential domain $D$ has the \textbf{unbounded ramification property} if for all $\gamma  \in \Q\Gamma_\infty$ with $\gamma > 0$ and $N \in \N$, there is some  $\gamma' \in \Gamma_n$ with $\gamma' > 0$ for some $n \geq N$ such that $\gamma' < \gamma$ considered as elements of $\Q\Gamma$.
\end{definition}

From Theorem \ref{Thm:IntROverNonMonicNonSingularValDomIsNotPrufer}, we know that if $V$ is a valuation domain with algebraically closed residue field and maximal ideal that is not principal, then $\IntR(V)$ is not a Prüfer domain. A sequential domain can spread out the obstructions that make the ring of integer-valued rational functions not a Prüfer domain. The unbounded ramification property mimics the property of having a maximal ideal that is not principal, even if all of the valuation overrings have a principal maximal ideal. There is an example following the proposition illustrating this phenomenon.

The following proposition uses the notation from the definition of a sequential domain. 
\begin{proposition}\label{Prop:SeqAlgClosedUnboundedRamificationNotPrufer}
	Let $D$ be a sequential domain with the unbounded ramification property. Also suppose that $V_\infty/\m_\infty$ is algebraically closed. Then $\IntR(D)$ is not a Prüfer domain. 
\end{proposition}

\begin{proof}
	We can assume that $D$ is Prüfer since if $D$ is not Prüfer, then $\IntR(D)$ is not Prüfer as well. We will suppose that $(x, \varpi) \subseteq \IntR(D)$ is invertible for a contradiction. Then there exist $\varphi, \psi \in (x, \varpi)^{-1}$ such that $\varphi \cdot x + \psi \cdot \varpi = 1$. Let $\alpha \coloneqq v_\infty(\varpi)$. There exists $a \in D$ with $v_\infty(a) = \alpha$ such that $\minval_{\varphi, v_\infty}(\alpha) = v_\infty(\varphi(a))$ and $\minval_{\psi, v_\infty}(\alpha) = v_\infty(\psi(a))$. Considering $\varphi(x)x + \psi(x)\varpi = 1$, we deduce that \[0 \geq \min\{v_\infty(\varphi(a)) + \alpha, v_\infty(\psi(a)) + \alpha\}.\] Because $\varphi, \psi \in (x, \varpi)^{-1}$, we have $v_\infty(\varphi(a)) + \alpha, v_\infty(\psi(a)) + \alpha\geq 0$. This means that $v_\infty(\varphi(a)) + \alpha= 0$ or $v_\infty(\psi(a)) + \alpha=0$. Either way, we have some $\rho \in (x, \varpi)^{-1}$ such that $\minval_{\rho, v_\infty}(\alpha) = - \alpha$. 
	
	Now for some $\epsilon \in \Q\Gamma_\infty$ with $\varepsilon > 0$, we have
	\[
	\minval_{\rho, v_\infty}(\gamma) = \begin{cases}
		c\gamma + \beta, & \alpha-\epsilon < \gamma \leq \alpha,\\
		c'\gamma + \beta', & \alpha \leq \gamma < \alpha + \epsilon,\\
	\end{cases}
	\]
	where $c, c' \in \Z$ and $\beta, \beta'\in \Gamma$. Since there are a finite number of coefficients which appear in $\rho$ and their values under $v_i$ is eventually constant and equal to their values under $v_\infty$, there exists $N \in \N$ such that for all $i \geq N$, we have $\minval_{\rho, v_\infty} = \minval_{\rho, v_i}$ and additionally $\alpha \in \Gamma_i$. Because $D$ has the unbounded ramification property, there exists some $n \geq N$ such that there exists $\delta \in \Gamma_n$ such that $0 < \delta < \epsilon$. Then $\alpha - \delta$ and $\alpha + \delta$ are both in $\Gamma_n$. Since $\rho \in (x, \varpi)^{-1}$ as a fractional ideal of $\IntR(D)$, we have $\minval_{\rho, v_n} \geq -\min\{\minval_{x, v_n}, \minval_{\varpi, v_n} \}$ on $(\Gamma_n)_{\geq 0}$. Combining this with the fact that $\alpha-\epsilon<\alpha - \delta< \alpha + \delta < \alpha + \epsilon$, we force $c \leq -1$ and $c' \geq 0$. Therefore, $c' - c > 0$. Also note that $\minval_{\rho, v_i}(\alpha) = -\alpha$ for all $i \geq N$.
	
	Write $\rho = \frac{f}{g}$ for some $f, g \in D[x]$. Since $V_\infty/\m_\infty$ is algebraically closed, $\loc_{f, \varpi, v_\infty}$ and $\loc_{g, \varpi, v_\infty}$ factor completely modulo $ \m_\infty$.  This means we can write
	\begin{align*}
		\frac{f(\varpi x)}{t} &= x^m(x-u_1)^{e_1} \cdots (x-u_r)^{e_r} + h_1(x),\\ \frac{g(\varpi x)}{t'}& = x^{m'}(x-u_1')^{e_1'} \cdots (x-u_r)^{e_r'} + h_2(x),
	\end{align*}
	for some $t, t' \in K$ with $v_\infty(t) = \minval_{f, v_\infty}(\alpha)$, $v_\infty(t') = \minval_{g, v_\infty}(\alpha)$, $u_1, \dots, u_r \in D$ representing distinct nonzero residues modulo $D \cap \m_\infty$, and $h_1, h_2 \in \m_\infty[x]$. We by Lemma \ref{Lem:PowersOfLoc} have that \[(e_1'-e_1) + \cdots + (e_r' - e_r) = c'- c > 0,\] so without loss of generality by permuting the indices, we can assume that $e_1' > e_1$. Write $h_1(x) = \sum b_jx^j$ and $h_2(x) = \sum b_j' x^j$ with each $b_j$ and $b_j'$ being in $ \m_\infty$. Furthermore, we have some $M \geq N$ such that for all $i \geq M$, we get $v_i(u_\ell) = 0$ and $v_i(u_\ell - u_{\ell'}) = 0$ for $\ell, \ell'$ distinct,  $v_i(b_j) = v_\infty(b_j) > 0$, and $v_i(b_j') = v_\infty(b_j') > 0$. Because $D$ has the unbounded ramification property, there exists $M' \geq M$ such that there exists $d \in D$ with $0 < v_{M'}(d) < \frac{1}{e_1'}\min\limits_j\{v_\infty(b_j), v_\infty(b_j') \}$. 
	
	Now we evaluate $\rho$ at $x = \varpi(d + u_1)$. We get

	\begin{align*}
		\rho(\varpi(d + u_1)) &= \frac{f(\varpi(d+u_1))}{g(\varpi(d+u_1))} \\&= \frac{t}{t'} \cdot \frac{d^{e_1}(d+u_1)^m(d+u_1-u_2)^{e_2} \cdots (d + u_1 - u_r)^{e_r} + h_1(d+u_1) }{d^{e_1'}(d+u_1)^{m'}(d+u_1-u_2)^{e_2'} \cdots (d+u_1 - u_r)^{e_r'} + h_2(d + u_1)}
	\end{align*}
	We have that $v_{M'}(d^{e_1}(d+u_1)^m(d+u_1-u_2)^{e_2} \cdots (d + u_1 - u_r)^{e_r}) = e_1v_{M'}(d)$ and $v_{M'}(h_1(d+u_1)) >  e_1v_{M'}(d)$ since each coefficient of the polynomial $h_1$ has coefficients with $v_{M'}$ valuation strictly greater than $e_1v_{M'}(d)$. Thus, \[v_{M'}(d^{e_1}(d+u_1)^m(d+u_1-u_2)^{e_2} \cdots (d + u_1 - u_r)^{e_r} + h_1(d+u_1)) = e_1v_{M'}(d).\] Similarly, \[v_{M'}(d^{e_1'}(d+u_1)^{m'}(d+u_1-u_2)^{e_2'} \cdots (d+u_1 - u_r)^{e_r'} + h_2(d + u_1)) = e_1'v_{M'}(d).\] This means that
	\begin{align*}
		v_{M'}(\rho(\varpi(d+u_1))) &= v_{M'}\left(\frac{t}{t'}\right) + e_1v_{M'}(d) - e_1'v_{M'}(d)\\ &= \minval_{\rho, v_{M'}}(\alpha) + (e_1-e_1')v_{M'}(d)\\& = -\alpha + (e_1-e_1')v_{M'}(d)\\& < -\alpha. 
	\end{align*}
	This is a contradiction since $\rho \cdot \varpi \in \IntR(D)$ but $v_{M'}(\rho(\varpi(d + u_1)) \varpi) < 0$. 
\end{proof}

\begin{example}
	Let $k$ be an uncountable algebraically closed field and form the field $K= k(s, t_1, t_2, \dots)$. Take $\{\alpha_1, \alpha_2, \dots\} \subseteq \R_{> 0}$ to be a $\Q$-linearly independent subset. For each $i \in \N$, define $v_i: K^\times \to \left(\sum\limits_{j\neq i} \Z\alpha_j \right) \oplus \Z \oplus \frac{1}{i}\Z$, ordered lexicographically, as the valuation such that 
	\[v_i(k^\times) = \{(0,0,0)\}, v_i(t_i) = \left(0,0, \frac{1}{i} \right), v_i(s) = (0,1,0), v_i(t_j) = (\alpha_j, 0,0),\]
	for $j \neq i$ and $v_i(f)$ for $f \in k[s, t_1, t_2, \dots]$ is the minimum of the $v_i$ values of each monomial of $f$. Then each $v_i$ extends uniquely to $K$. Let $V_i$ be the associated valuation domain of $v_i$. Define the domain $D = \bigcap\limits_{i =1 }^\infty V_i$. We see that $D$ is a sequential domain with field of fractions $K$ (since $k \subseteq D$ and $s, t_1, t_2, \dots \in D$) where each $V_i$ has an algebraically closed residue field and a principal maximal ideal, generated by $t_i$. Additionally, each value group embeds naturally into $\left(\sum\limits_{i = 1}^\infty \Z\alpha_j \right) \oplus \Z \oplus \Q$, and $v_\infty$ gives rise to the associated valuation overring $V_\infty$, which also has algebraically closed residue field and principal maximal ideal, generated by $s$. This domain $D$ also has the unbounded ramification property, since $v_i(t_i)$ can be arbitrarily small. We also know that $D$ is a Bézout domain and therefore a Prüfer domain by \cite[Theorem 6.6]{OlberdingRoitman} since $D$ is the intersection of countably many valuation domains containing $k$, an uncountable subfield. However, by the previous proposition, $\IntR(D)$ is not a Prüfer domain.
	
	Notice that each $V_i$ for $i \in \N \setminus \{0\}$ has a principal maximal ideal. The same holds true for $V_\infty$. Thus, Corollary \ref{Cor:ValOverringObstruction} cannot be used immediately to show that $\IntR(D)$ is not a Prüfer domain.  
\end{example}

\subsection{Intersection of monic and singular Prüfer domains}

We know that if $D$ is a Prüfer domain that is monic or singular, then $\IntR(D)$ is also a Prüfer domain. We want to see if there are other instances of Prüfer domains whose ring of integer-valued rational functions is also a Prüfer domain. For this reason, we consider Prüfer domains $D$
that are neither monic nor singular. Some domains of this form are formed by the intersection of a monic Prüfer domain and a singular Prüfer domain. 

We first consider a family of Prüfer domains obtained from intersecting a monic Prüfer domain and a singular Prüfer domain with some extra conditions. We then give some corollaries with conditions that are easier to verify. This includes the case of a finite intersection of valuation domains. Examples will follow to showcase the different conditions.

\begin{theorem}\label{Thm:MonicSingular}
	Let $D$ be a Prüfer domain with $K$ as the field of fractions. Suppose that $D$ can be written as $D = D_1 \cap D_2$, where $D_1$ is a monic Prüfer overring of $D$ and $D_2$ is a singular Prüfer overring of $D$. Suppose there exist $n \in \N$, a collection $\{V_\lambda\}$ of valuation overrings of $D_1$, and a collection $\{W_\mu\}$ of valuation overrings of $D_2$ such that

	\begin{itemize}
		\item $D_1 = \bigcap\limits_\lambda V_\lambda$ and $D_2 = \bigcap\limits_{\mu} W_\mu$,
		\item  the maximal ideal of $W_\mu$ is generated by some $\varpi_\mu \in W_\mu$,
		\item there exists some $d \in D_2$ such that $0 < w_\mu(d) < nw_\mu(\varpi_\mu)$ for all $\mu$, where $w_\mu$ is the valuation corresponding to $W_\mu$,
		\item $v_\lambda(d) = 0$ for all $\lambda$, where $v_\lambda$ is the valuation corresponding to $V_\lambda$, and
		\item there exist polynomials $f, g, h_1, h_2 \in D[x]$ such that 
		\begin{itemize}
			\item  $f$ is monic of degree $n$ such that $f(D_1) \subseteq D_1^\times$, 
			\item $g(D_1) \subseteq D_1^\times$, $g$ has leading coefficient $d$, $\deg g = n$, and $w_\mu(g(a)) = w_\mu(da^n)$ for all $\mu$ and for all $a \in K$ such that $w_\mu(a) < 0$, 
			\item $h_1$ is monic and $\deg h_1 = n$, $\deg h_2 < n$, and
			\item $dfh_1 + gh_2$ is unit-valued for each $V_\lambda$ and $W_\mu$. 
		\end{itemize}
	\end{itemize}
	Then $\IntR(E, D)$ is Prüfer domain with torsion Picard group for any subset $E \subseteq K$. 
\end{theorem}

\begin{proof}
	Fix a subset $E \subseteq K$. 
	
	Set $\theta(x) = \frac{df(x)}{g(x)}$. Fix $\mu$. Let $a \in K$. If $w_\mu(a) \geq 0$, then  $w_\mu((dfh_1 + gh_2)(a)) = 0$, which implies $w_\mu((gh_2)(a)) = 0$ and thus $w_\mu(g(a)) = 0$. Then we calculate
	\[
	w_\mu(\theta(a)) = \begin{cases}
		w_\mu(df(a)) \geq 0, & \text{if $w_\mu(a) \geq 0$},\\
		w_\mu(da^n) - w_\mu(da^n) =0, & \text{if $w_\mu(a) < 0$}.
	\end{cases}
	\]
	Similarly, if we fixed $\lambda$, then since $g(D_1) \subseteq D_1^\times$ and the leading coefficient of $g$ is a unit of $D_1$, we know that $g$ is unit-valued for $V_\lambda$ \cite[Proposition 2.1]{PruferNonDRings}. Let $a \in K$, we have that
	\[
	v_\lambda(\theta(a)) = \begin{cases}
		v_\lambda(df(a)) = 0, & \text{if $v_\lambda(a) \geq 0$},\\
		v_\lambda(da^n) - v_\lambda(da^n) =0, & \text{if $v_\lambda(a) < 0$}.
	\end{cases}
	\]
	This shows that $\theta \in \IntR(K,D)$. 
	
	Now let $a, b \in K$ with $b \neq 0$. We define 
	\[a \diamond b \coloneqq \theta\left(\frac{a}{b} \right)b^n h_1\left(\frac{a}{b}\right) +b^n h_2\left(\frac{a}{b}\right).\]
	We claim that $v(a \diamond b) = \min\{v(a^n), v(b^n) \}$ for any valuation $v = v_\lambda$ or $v = w_\mu$. 
	
	Fix a valuation $v = v_\lambda$ or $v = w_\mu$. Let $a, b \in K$ with $b \neq 0$. We have that 
	\[
	\frac{a \diamond b}{b^n} = \theta\left(\frac{a}{b} \right) h_1\left(\frac{a}{b}\right) + h_2\left(\frac{a}{b}\right) = \left(\frac{dfh_1 + gh_2}{g}\right)\left(\frac{a}{b}\right),
	\]
	so if $v(a) \geq v(b)$, then $v(a \diamond b) = v(b^n)$ since both $dfh_1 + gh_2$ and $g$ are unit-valued over the valuation ring associated with $v$. 
	
	Now suppose that $v(a) < v(b)$. Write $h_1(x) = a_0 + a_1x + \cdots + a_{n-1}x^{n-1} + x^n$ and $h_2(x) = b_0 + b_1x + \cdots + b_{r}x^{r}$ with $a_i, b_i \in D$ and some $r < n$. Then we have
	\[
	a \diamond b \coloneqq \theta\left(\frac{a}{b} \right)(a_0b^n + a_1ab^{n-1} + \cdots + a_{n-1}a^{n-1}b + a^n) +(b_0b^n + b_1ab^{n-1} + \cdots + b_{r}a^{r}b^{n-r}).
	\]
	If $v(a) < v(b)$, then $v\left(\theta\left(\frac{a}{b} \right)\right) = 0$, so $v(a \diamond b) = v(a^n)$, as desired. 
	
	Now take $\varphi, \psi \in \IntR(E, D)$ with $\psi \neq 0$. We claim that $(\varphi, \psi)^n$ is generated by
	\[
	\rho \coloneqq \theta\left(\frac{\varphi}{\psi} \right)(a_0\psi^n + a_1\varphi\psi^{n-1} + \cdots + a_{n-1}\varphi^{n-1}\psi + \varphi^n) +(b_0\psi^n + b_1\varphi\psi^{n-1} + \cdots + b_{r}\varphi^{r}\psi^{n-r}).
	\]
	We see that $\rho \in (\varphi, \psi)^n$ since $\theta \in \IntR(K,D)$ and $a_i, b_i \in D$. Furthermore, let $j, k \in \N$ such that $j+k = n$. Then for each $a \in K$ such that $\psi(a) \neq 0$, we have \[v(\rho(a)) = v(\varphi(a) \diamond \psi(a)) = \min\{v(\varphi(a))^n, v(\psi(a))^n\} \leq v(\varphi(a)^j\psi(a)^k)\] for all valuations $v = v_\lambda$ or $w_\mu$. This implies that $\rho$ divides $\varphi^j\psi^k$ in $\IntR(K,D)$ and therefore also in $\IntR(E,D)$. Thus, $(\varphi, \psi)^n \subseteq (\rho)$. We then get that $(\varphi, \psi)^n = (\rho)$. 
	
	Since a power of an ideal of $\IntR(E,D)$ generated by two elements is principal, we also know that this ideal is invertible. Thus, $\IntR(E,D)$ is a Prüfer domain. 
	
	Now let $(\varphi_1, \dots, \varphi_m)$ be a finitely-generated, and thus invertible, ideal of $\IntR(E,D)$. We can ensure that $\varphi_1, \dots, \varphi_m$ are all nonzero. Then as before, for each $a \in K$ except for the finitely many values such that $a$ is a pole for some $\varphi_i$, we see that 
	\[
	v((((\varphi_1(a) \diamond \varphi_2(a) ) \diamond \varphi_3(a)^n) \diamond \varphi_4(a)^{n^2}) \diamond \cdots \diamond \varphi_m(a)^{n^{m-2}} ) = \min_i\{ n^{m-1}v(\varphi_i(a)) \}
	\]
	for all valuations $v = v_\lambda$ or $v=w_\mu$. Using the same arguments as before, it follows that $(\varphi_1, \dots, \varphi_m)^{n^{m-1}}$ is principal. Thus, the Picard group of $\IntR(E,D)$ is torsion.
\end{proof}

\begin{corollary}\label{Cor:MonicSingularApproximation}
	Let $D$ be a Prüfer domain with $K$ as the field of fractions. Suppose that $D$ can be written as $D = D_1 \cap D_2$, where $D_1$ is a monic Prüfer overring of $D$ and $D_2$ is a singular Prüfer overring of $D$. Suppose there exist $n \in \N$, a collection $\{V_\lambda\}$ of valuation overrings of $D_1$, and a collection $\{W_\mu\}$ of valuation overrings of $D_2$ such that

	\begin{itemize}
		\item $D_1 = \bigcap\limits_\lambda V_\lambda$ and $D_2 = \bigcap\limits_{\mu} W_\mu$,
		\item  the maximal ideal of $W_\mu$ is generated by some $\varpi_\mu \in W_\mu$,
		\item there exists some $d \in D_2$ such that $0 < w_\mu(d) < nw_\mu(\varpi_\mu)$ for all $\mu$, where $w_\mu$ is the valuation corresponding to $W_\mu$,
		\item $v_\lambda(d-1) > 0$ for all $\lambda$, where $v_\lambda$ is the valuation corresponding to $V_\lambda$, and
		\item there exist a monic polynomial $f \in D[x]$ of degree $n$ such that $f(D_1) \subseteq D_1^\times$ and $f(0) \in D^\times$.
	\end{itemize}
	Then $\IntR(E, D)$ is Prüfer domain with torsion Picard group for any subset $E \subseteq K$. 
\end{corollary}

\begin{proof}
	We verify the conditions of Theorem \ref{Thm:MonicSingular}.
	
	We see that $v_\lambda(d) = 0$ for all $\lambda$ from the fact that $v_\lambda(d-1) > 0$. Furthermore, $g(x) \coloneqq d(f(x)-f(0)) + f(0)$ is unit-valued over $D_1$. This is because for any $a \in D_1$ and any $\lambda$, we have that $v_\lambda(g(a) - f(a)) = v_\lambda((d-1)(f(a) - f(0))) > 0$, so $v_\lambda(g(a)) = v_\lambda(f(a)) = 0$. Observe that the leading coefficient of $g$ is $d$. Also, fix a $\mu$ and let $a \in K$ such that $w_\mu(a) < 0$. Then $w_\mu(g(a)) = w_\mu(d(f(a)-f(0)) +f(0)) = w_\mu(da^n)$ since $w_\mu(f(a)-f(0)) = w_\mu(a^n)$. 
	
	Lastly, we set $h_1(x) \coloneqq f(x) - f(0)$ and $h_2(x) \coloneqq f(0)$. We have $n = \deg h_1 > \deg h_2$ and $h_1$ is monic. Next, we must check that $dfh_1 + gh_2$ is unit-valued for all $V_\lambda$ and all $W_\mu$. We have
	\begin{align*}
		(dfh_1 + gh_2)(x) &= df(x)(f(x) - f(0)) + (d(f(x) - f(0)) + f(0))f(0) \\&= df(x)^2 -df(x)f(0) + df(x)f(0) -df(0)^2 + f(0)^2 \\
		& = df(x)^2 -df(0)^2 + f(0)^2.
	\end{align*}
	For any $a \in D$, we have $w_\mu((dfh_1 + gh_2)(a)) = w_\mu(f(0)^2) = 0$ for all $\mu$ and 
	\[
	v_\lambda((dfh_1 + gh_2)(a) - f(a)^2) = v_\lambda((d-1)(f(a)^2-f(0))) > 0
	\]
	so $v_\lambda((dfh_1 + gh_2)(a)) = 0$ for all $\lambda$. The domain $D$ satisfies all of the hypotheses of Theorem \ref{Thm:MonicSingular}, so  $\IntR(E, D)$ is Prüfer domain with torsion Picard group for any subset $E \subseteq K$. 
\end{proof}

The intersection of finitely many valuation domains with the same field of fractions is a Prüfer domain \cite[Theorem 22.8]{Gilmer}. If these valuation domains are pairwise independent, we many use the approximation theorem for independent valuations \cite[Theorem 22.9]{Gilmer} to obtain the following result.

\begin{corollary}
	Let $V_1, \dots, V_r$ be pairwise independent valuation rings on the field $K$. Suppose that each $V_i$ has a principal maximal ideal or a residue field that is not algebraically closed and set $D = V_1 \cap \cdots \cap V_n$. Then $\IntR(E,D)$ is a Prüfer domain with torsion Picard group for every subset $E$ of $K$. 
\end{corollary} 

\begin{proof}
	We can order the indices so that $V_1, \dots, V_s$ for some $s \leq r$ are such that for each $i = 1, \dots, s$, the residue field of $V_i$ is not algebraically closed and the residue fields for $V_{s+1}, \dots, V_r$ are algebraically closed. For $i \leq s$, this means there is some monic nonconstant unit valued polynomial $f_i$ for $V_i$. Taking appropriate powers of each of the $f_i$, we can assume that all of the $f_i$ have the same degree $n$. Now for each $i = 1, \dots, s$, write
	\[
	f_i(x) = x^n + \sum_{j=0}^{n-1} a_{ij} x^j,
	\]
	where each $a_{ij} \in V_i$. Then by the approximation theorem for pairwise independent valuations, we have that for each $j = 0,1, \dots, n-1$ an element $A_j \in K$ such that
	\[
	v_i(A_j - a_{ij}) > 0 \text{ for all $i = 1, \dots, s$ and } v_i(A_j) = 0 \text{ for all $i > s$},
	\]
	where $v_i$ denotes the valuation corresponding to $V_i$. Form
	\[
	F(x) \coloneqq x^n + \sum_{j=0}^{n-1} A_j x^j. 
	\]
	We claim that $F \in D[x]$ is unit-valued for $D_1 \coloneqq V_1 \cap \cdots \cap V_s$. Let $a \in D_1$ and let $i \in \{1, \dots, s\}$. Then 
	\[
	v_i(F(a)-f_i(a)) = v_i\left(\sum_{j=0}^{n-1} (A_j-a_{ij})a^j\right) > 0
	\]
	and $v_i(f_i(a)) = 0$ imply that $v_i(F(a)) = 0$. Therefore, $F(a) \in D_1^\times$. Furthermore, $F(0) = A_0$ and $v_i(A_0) = 0$ for all $i=1, \dots, r$, so $F(0) \in D^\times$. 
	
	Now let $D_2 = V_{s+1} \cap \cdots \cap V_{r}$. For $i = s+1, s+2, \dots, r$, the residue field of $V_i$ is algebraically closed so the maximal ideal of $V_i$ must be principal by hypothesis. Say the maximal ideal of $V_i$ is generated by some $\varpi_i \in V_i$. Then by the approximation theorem again, we have some $d \in K$ such that
	\[
	v_i(d-1) > 0 \text{ for $i \leq s$ and } v_i(d) = v_i(\varpi_i) \text{ for $i > s$}.
	\]
	Then we verify that $d \in D_2$ and $0 < v_i(d) = v_i(\varpi_i) < nv_i(\varpi)$ for all $i > s$. This means $D$ satisfies all of the hypotheses of the previous corollary and thus $\IntR(E, D)$ is a Prüfer domain with torsion Picard group for any subset $E \subseteq K$. 
\end{proof}

Here is an example of a Prüfer domain $D$ with a finite number of maximal ideals that is neither singular nor monic. The previous corollary allows us to determine that $\IntR(E, D)$ is a Prüfer domain for any subset $E \subseteq K$. 

\begin{example}
	Let $K = k(s,t)$, where $k$ is any real closed field. Also, let $\alpha \in \R$ with $\alpha > 0$ be irrational. Define a valuation $v_1:K \to \Z[\alpha] \cup \{\infty\}$ as follows. We first define $v_1$ on $k[s,t]\setminus\{0\}$ by
	\[
	v_1\left(\sum a_{i_1i_2} s^{i_1} t^{i_2}\right) = \min\{i_1+i_2\alpha \mid a_{i_1i_2} \neq 0 \},
	\]
	where $\sum a_{i_1i_2} s^{i_1} t^{i_2} \in k[s,t]\setminus\{0\}$.
	Then this function uniquely extends to a valuation on $K$. 
	
	Now we define a valuation $v_2:K \to \Z \oplus \Z \cup \{\infty\}$, ordered lexicographically. We first define another valuation $w_2:k(s) \to \Z \cup \{\infty\}$ as the valuation corresponding to the valuation ring $k[s]_{(s)}$ so that $w_2(s) = 1$. Let $\Delta \coloneqq k[s]_{(s)}/sk[s]_{(s)}$, the residue field of $k[s]_{(s)}$. Then define $w_2':\Delta(t) \to \Z \cup \{\infty\}$ as the valuation corresponding to $\Delta[t]_{(t^2+1)}$ so that $w_2'(t^2+1) = 1$. Now we can define $v_2$ on $k[s]_{(s)}[t] \setminus \{0\}$. Let $f(t) \in k[s]_{(s)}[t]$. We write $f(t) = dg(t)$ for some $d \in k[s]_{(s)}$ and $g(t) \in k[s]_{(s)}[t] \setminus sk[s]_{(s)}[t]$. Now, we have
	\[
	v_2(f) = (w_2(d), w_2'(g \mod sk[s]_{(s)})), 
	\]
	which extends uniquely to a valuation on $K$. 
	
	Let $V_1$ and $V_2$ be the valuation rings corresponding to $v_1$ and $v_2$, respectively. Then $V_1$ has residue field that isomorphic to $k$, which is not algebraically closed, and the maximal ideal is not principal. As for $V_2$, the ring $V_2$ has residue field isomorphic to $k[t]/(t^2+1)$, which is algebraically closed, and the maximal ideal is principal. Therefore, $D = V_1 \cap V_2$ is a Prüfer domain that is neither monic nor singular, but $V_1$ and $V_2$ are independent valuations on $K$, so $\IntR(E, D)$ is Prüfer for any subset $E$ of $K$ by the previous corollary. 
\end{example}

The next example is a Prüfer domain $D$ that is not the intersection of finitely many valuation domains, but we can use Corollary \ref{Cor:MonicSingularApproximation} to determine that $\IntR(E,D)$ is Prüfer. In the next two examples, we make use of the fact that the intersection of a countable number of valuation domains with the same field of fractions all containing a common uncountable field is a Bézout domain and thus a Prüfer domain \cite[Theorem 6.6]{OlberdingRoitman}.

\begin{example}
	Let $K = k(t_1, t_2, \dots)$, where $k$ is an uncountable algebraically closed field. Also let $\{\alpha_1, \alpha_2, \dots \} \subseteq \R_{> 0}$ be a $\Q$-linearly independent subset of $\R$. For $i \in \N \setminus \{0\}$, we define
	\begin{align*}
		v_i\left(\sum a_{e_1 e_2 \cdots} t_1^{e_1} t_2^{e_2} \cdots \right)& = \min\left\{ \sum_{j=2}^\infty e_j \alpha_{i+j} \,\middle\vert\, a_{e_1 e_2 \cdots}  \neq 0 \right\},\\
		w_i\left(\sum a_{e_1 e_2 \cdots}  a_{e_1 e_2 \cdots} t_1^{e_1} (t_2+1)^{e_2} t_3^{e_3} t_4^{e_4} \cdots \right) &= \min\left\{ \left( \sum_{j \neq 2} e_j \alpha_{i+j} , e_2 \right) \,\middle\vert\, a_{e_1 e_2 \cdots}  \neq 0 \right\},
	\end{align*}
	where both sums on the left range over $(e_1, e_2 ,\dots) \in \bigoplus\limits_{\ell=1}^\infty \N$ and each $a_{e_1 e_2 \cdots} \in k$ with all but finitely many $a_{e_1 e_2 \cdots}$ are $0$. The value group of each $w_i$ is $\left(\sum\limits_{j= 2}\Z \alpha_{i+j}\right) \oplus \Z$ endowed with left-to-right lexicographic ordering. These functions extend uniquely to valuations on $K$, since $K$ is the field of fractions of $k[t_1, t_2, \dots] = k[t_1, t_2+1, t_3, t_4, \dots]$. Let $V_i$ and $W_i$ be the valuation domains corresponding to $v_i$ and $w_i$, respectively. 
	
	For the valuation $v_i:K \to \sum\limits_{j=2}^{\infty}\Z \alpha_{i+j} \cup \{\infty\}$, $V_i$ has residue field isomorphic to $k(t)$, which is not algebraically closed, and the maximal ideal of $V_i$ is not principal. 
	
	As for $w_i: K \to \left(\sum\limits_{j\neq 2}\Z \alpha_{i+j}\right) \oplus \Z \cup \{\infty\}$, we have $W_i$ having residue field isomorphic to $k$, which is algebraically closed and the maximal ideal of $W_i$ is principal, generated by $t_2+1$. 
	
	Set $D_1 \coloneqq \bigcap\limits_{i=1}^\infty V_i$. We know that $D_1$ is a Prüfer domain since $D_1$ is the intersection of countably many valuation domains with a common uncountable subfield $k$. We have that $x^2 -(t_1 +1)$ is unit valued for $D_1$, so $D_1$ is monic. Furthermore, $K$ is the field of fractions of $D_1$ since $k \subseteq D_1$ and $t_1, t_2, \dots \in D_1$. 
	
	Set $D_2 \coloneqq \bigcap\limits_{i=1}^\infty W_i$. As with $D_1$, we can verify that $D_2$ is a Prüfer domain with field of fractions $K$. We also have $0 < w_i(t_2+1) < 2w_i(t_2+1)$, so $D_2$ is singular. 
	
	Now form $D \coloneqq D_1 \cap D_2$. Since $k \subseteq D$ and $t_1, t_2, \dots \in D$, we have that $K$ is the field of fractions of $D$ as well. Note that $D$ is also the intersection of countably many valuation domains with a common uncountable subfield $k$, so $D$ is Prüfer. 
	
	Now we verify the remaining conditions for $D$ for Corollary \ref{Cor:MonicSingularApproximation}. Here, we'll use $n = 2$. We have that $f(x)\coloneqq x^2 - (t_1 +1) \in D[x]$. Also, $v_i(-(t_1 +1)) = 0$ and $w_i(-(t_1 +1)) = 0$ for all $i$, so $f(0) \in D^\times$. Plus, $v_i((t_2+1)-1) = v_i(t_2) = \alpha_{i+2} > 0$ for all $i$. Therefore, by the corollary, we have that $\IntR(E,D)$ is Prüfer for any subset $E$ of $K$. 
\end{example}

The full power of Theorem \ref{Thm:MonicSingular} is used to determine that $\IntR(E,D)$ is Prüfer in the following example.

\begin{example}
	We let $K = k(t_1, t_2, \dots)$, where $k$ is an uncountable algebraically closed field. We also take $\{\alpha_1, \alpha_2, \dots \} \subseteq \R_{> 0}$ to be a $\Q$-linearly independent subset of $\R$. For $i \in \N \setminus\{0\}$, we define
	\begin{align*}
		v_i\left(\sum a_{e_1 e_2 \cdots} t_1^{e_1} t_2^{e_2} \cdots \right) &= \min\left\{ \sum_{j=i+1}^\infty e_j \alpha_{i+j} \,\middle\vert\, a_{e_1 e_2 \cdots}  \neq 0 \right\},\\
		w_i\left(\sum a_{e_1 e_2 \cdots} t_1^{e_1} t_2^{e_2}  \cdots \right) &= \min\left\{ \left( \sum_{j = 2}^\infty e_j \alpha_{i+j} , e_1 \right) \,\middle\vert\, a_{e_1 e_2 \cdots}  \neq 0 \right\},
	\end{align*}
	where both sums on the left range over $(e_1, e_2 ,\dots) \in \bigoplus\limits_{\ell=1}^\infty \N$ and each $a_{e_1 e_2 \cdots} \in k$ with all but finitely many $a_{e_1 e_2 \cdots}$ are $0$. The value group of each $w_i$ is $\left(\sum\limits_{j= 2}\Z \alpha_{i+j}\right) \oplus \Z$ endowed with left-to-right lexicographic ordering. These functions extend uniquely to valuations on $K$. Let $V_i$ and $W_i$ be the valuation rings corresponding to $v_i$ and $w_i$, respectively.

	For the valuation $v_i:K \to \sum\limits_{j=i+1}^{\infty}\Z \alpha_{i+j} \cup \{\infty\}$, $V_i$ has residue field isomorphic to $k(t_1, \dots, t_i)$, which is not algebraically closed, and the maximal ideal of $V_i$ is not principal. 
	
	As for $w_i: K \to \left(\sum\limits_{j= 2}\Z \alpha_{i+j}\right) \oplus \Z \cup \{\infty\}$, we have $W_i$ having residue field isomorphic to $k$, which is algebraically closed and the maximal ideal of $W_i$ is principal, generated by $t_1$. 
	
	Set $D_1 \coloneqq \bigcap\limits_{i=1}^\infty V_i$, $D_2 \coloneqq \bigcap\limits_{i=1}^\infty W_i$, and $D \coloneqq D_1 \cap D_2$. Since $D_1, D_2$, and $D$ are all countable intersections of valuation domains all containing $k$, an uncountable field, we know that $D_1, D_2$, and $D$ are all Prüfer domains. Furthermore, $t_1, t_2, \dots, \in D$ and $k \subseteq D$, so $K$ is the field of fractions $D_1, D_2,$ and $D$. We see that $x^2 - t_1$ is unit-valued for $D_1$, so $D_1$ is monic. We also have $0 < w_i(t_1) < 2w_i(t_1)$, so $D_2$ is singular. 
	
	There does not exist $d \in D_2$ and $n \in \N$ such that $0 < w_i(d) < nw_i(t_1)$ for all $i$ and $v_i(d-1) > 0$ for all $i$. This is because there exists $i$ large enough such that $v_i(d-1) = 0$ for any choice of $d$. Therefore, the conditions in Corollary \ref{Cor:MonicSingularApproximation} are not satisfied. 
	
	Nevertheless, we can use Theorem \ref{Thm:MonicSingular}. Set $n=2, d = t_1, f(x) = x^2-t_1$, $g(x) = t_1x^2 -1$, $h_1(x) = x^2$ and $h_2(x) = 1$. We know that $v_i(t_1) = 0$ for all $i$. We verify that $x^2 - t_1$ is unit-valued over $D_1$ since $x^2-t_1$ has no roots over the residue field $k(t_1, t_2, \dots, t_i)$ of $V_i$ for any $i$. Similarly, $g(x)$ is unit-valued over $D_1$ for the same reason. Also, $\deg g = 2$ and for all $a \in K$ and $i$ such that $w_i(g(a))) < 0$, we have $w_i(t_1a^2)$. 
	
	Lastly, we check that
	\[
	dfh_1 + gh_2 = t_1x^4 + (t_1-t_1^2)x^2 - 1
	\]
	is unit-valued over each $V_i$ and each $W_i$. Fix an $i$. Let $\m_i$ be the maximal ideal of $V_i$. Since $V_i/\m_i$ is isomorphic to $k(t_1, \dots, t_i)$, we let $v$ be the valuation corresponding to $k(t_2, \dots, t_i)[t_1]_{(t_1)}$. If $t_1x^4 + (t_1-t_1^2)x^2 - 1 \mod \m_i$ has a root $\xi$, then $v(\xi) = -\frac{v(t_1)}{4}$, which is impossible. Lastly, notice that $dfh_1 + gh_2$ is equivalent to $-1$ modulo the maximal ideal of any $W_i$, which makes $dfh_1 + gh_2$ unit-valued for each $W_i$. Thus, $\IntR(E,D)$ is Prüfer for any subset $E$ of $K$ by Theorem \ref{Thm:MonicSingular}. 
\end{example}

\bibliographystyle{amsalpha}
\bibliography{references}
\end{document}